\documentclass[
paper=letter,%
numbers=noendperiod,%
captions=nooneline,%
abstracton,%
DIV=10%
]{scrartcl}

\usepackage[T1]{fontenc}%
\usepackage{lmodern}%
\usepackage[american]{babel}%
\usepackage{microtype}%

\usepackage[
hyperref,%
table%
]{xcolor}%

\usepackage{scrlayer-scrpage}%

\usepackage{eso-pic}%
\usepackage{rotating}%
\usepackage{amsmath}%
\usepackage{mathtools}%
\usepackage{amssymb}%
\usepackage{amsthm}%
\usepackage{thmtools}%
\usepackage{etoolbox}%
\usepackage{bm}%
\usepackage{bbm}%
\usepackage{enumitem}%
\usepackage{graphicx}%
\usepackage{grffile}%
\usepackage{tikz}%
\usepackage{wrapfig}%
\usepackage{tabularx}%
\usepackage{siunitx}%
\usepackage{booktabs}%
\usepackage{multirow}%
\usepackage{vruler}%
\usepackage{fancyvrb}%
\usepackage{listings}%
\usepackage{csquotes}%
\usepackage[
style=authoryear,%
dashed=false,%
hyperref=true,%
useprefix=true,%
maxnames=2,%
maxbibnames=6,%
uniquename=false,%
]{biblatex}%
\usepackage[
hypertexnames=false,%
setpagesize=false,%
pdfborder={0 0 0},%
pdfstartview=Fit,%
bookmarksopen=true,%
bookmarksnumbered=true%
]{hyperref}%
\xdefinecolor{lightgray}{RGB}{247, 247, 247}%
\xdefinecolor{semilightgray}{RGB}{240, 240, 240}%
\xdefinecolor{middlegray}{RGB}{127, 127, 127}%
\xdefinecolor{blue}{RGB}{58, 95, 205}%
\xdefinecolor{deepskyblue}{RGB}{0, 154, 205}%
\xdefinecolor{chocolate}{RGB}{205, 102, 29}%

\setkomafont{pageheadfoot}{\normalfont\normalcolor\sffamily}%
\setkomafont{pagenumber}{\normalfont\normalcolor\sffamily}%
\automark{section}%
\setkomafont{captionlabel}{\normalfont\normalcolor\sffamily\bfseries}%

\newcommand*{\mysquare}{\rule[0.18em]{0.36em}{0.36em}}
\newcommand*{\mytriangle}{\raisebox{0.12em}{\resizebox{0.48em}{0.48em}{$\blacktriangleright$}}}
\newcommand*{\mybar}{\rule[0.32em]{0.62em}{0.08em}}
\newcommand*{\mydot}{\raisebox{0.14em}{\resizebox{0.44em}{!}{$\bullet$}}}
\setlist{%
  align=left,%
  labelindent=0mm, %
  leftmargin=!,%
  itemindent=0mm, %
  listparindent=\parindent,%
  parsep=0mm,%
  topsep=1mm,%
  itemsep=1mm%
}
\setlist[itemize,1]{label={\mysquare\ }, labelwidth=\widthof{\mysquare\ }}%
\setlist[itemize,2]{label={\mytriangle\ }, labelwidth=\widthof{\mytriangle\ }}%
\setlist[itemize,3]{label={\mybar\ }, labelwidth=\widthof{\mybar\ }}%
\setlist[itemize,4]{label={\mydot\ }, labelwidth=\widthof{\mydot\ }}%
\setlist[enumerate,1]{label=\arabic*), labelwidth=\widthof{9)}}%
\setlist[enumerate,2]{label=\arabic{enumi}.\arabic*), labelwidth=\widthof{9.9)}}%
\setlist[enumerate,3]{label=\arabic{enumi}.\arabic{enumii}.\arabic*), labelwidth=\widthof{9.9.9)}}%

\makeatletter
\newcommand\myisodate{\number\year-\ifcase\month\or 01\or 02\or 03\or 04\or 05\or 06\or 07\or 08\or 09\or 10\or 11\or 12\fi-\ifcase\day\or 01\or 02\or 03\or 04\or 05\or 06\or 07\or 08\or 09\or 10\or 11\or 12\or 13\or 14\or 15\or 16\or 17\or 18\or 19\or 20\or 21\or 22\or 23\or 24\or 25\or 26\or 27\or 28\or 29\or 30\or 31\fi}%
\makeatother
\newcommand*{\abstractnoindent}{}%
\let\abstractnoindent\abstract
\renewcommand*{\abstract}{\let\quotation\quote\let\endquotation\endquote
  \abstractnoindent}
\deffootnote[1em]{1em}{1em}{\textsuperscript{\thefootnotemark}}%
\renewcommand*{\big}[1]{{\vcenter{\hbox{\scalebox{1.30}{\ensuremath#1}}}}}%

\lstdefinestyle{input}{
  backgroundcolor=\color{semilightgray},%
  commentstyle=\itshape\color{chocolate},%
  keywordstyle=\color{blue},%
  stringstyle=\color{blue},%
  numbers=left,%
  numbersep=4.8pt,%
  numberstyle=\color{darkgray!80}\tiny%
}
\lstdefinestyle{output}{
  backgroundcolor=\color{lightgray}%
}

\lstdefinestyle{Lstyle}{
  language=[LaTeX]TeX,%
  texcs={},%
  otherkeywords={}%
}

\lstnewenvironment{Linput}[1][]{%
  \lstset{style=input, style=Lstyle}
  #1%
}{\vspace{-0.25\baselineskip}}%

\lstnewenvironment{Loutput}[1][]{%
  \lstset{style=output, style=Lstyle}
  #1%
}{\vspace{-0.25\baselineskip}}%

\lstdefinestyle{Rstyle}{
  language=R,%
  keywords={if, else, repeat, while, function, for, in, next, break},%
  otherkeywords={}%
}

\lstnewenvironment{Rinput}[1][]{%
  \lstset{style=input, style=Rstyle}
  #1%
}{\vspace{-0.25\baselineskip}}%

\lstnewenvironment{Routput}[1][]{%
  \lstset{style=output, style=Rstyle}
  #1%
}{\vspace{-0.25\baselineskip}}%

\DeclareNameAlias{sortname}{last-first}%
\DefineBibliographyExtras{american}{\DeclareQuotePunctuation{}}%
\renewbibmacro*{volume+number+eid}{%
  \setunit*{\addcomma\space}%
  \printfield{volume}%
  \printfield{number}}
\DeclareFieldFormat*{number}{(#1)}
\DeclareFieldFormat*{title}{#1}%
\DeclareFieldFormat{doi}{%
  \ifhyperref
    {\href{http://dx.doi.org/#1}{\nolinkurl{doi:#1}}}%
    {\nolinkurl{doi:#1}}}%
\renewbibmacro*{in:}{}%
\DeclareFieldFormat{isbn}{ISBN #1}%
\DeclareFieldFormat{pages}{#1}%
\DeclareFieldFormat{url}{\url{#1}}%
\DeclareFieldFormat{urldate}{\mkbibparens{#1}}%
\renewcommand*{\cite}[2][]{\textcite[#1]{#2}}%

\newif\ifstarttheorem
\declaretheoremstyle[%
  spaceabove=0.5em,
  spacebelow=0.5em,
  headfont=\sffamily\bfseries\global\starttheoremtrue,
  notefont=\sffamily\bfseries,
  notebraces={(}{)},
  headpunct={},
  bodyfont=\normalfont,
  postheadspace=\newline%
]{myMainStyle}
\declaretheorem[style=myMainStyle, numberwithin=section]{definition}%
\declaretheorem[style=myMainStyle, sibling=definition]{proposition}
\declaretheorem[style=myMainStyle, sibling=definition]{lemma}
\declaretheorem[style=myMainStyle, sibling=definition]{theorem}
\declaretheorem[style=myMainStyle, sibling=definition]{corollary}
\declaretheorem[style=myMainStyle, sibling=definition]{remark}
\declaretheorem[style=myMainStyle, sibling=definition]{example}
\declaretheorem[style=myMainStyle, sibling=definition]{algorithm}

\makeatletter
\preto\itemize{%
  \if@inlabel
    \ifstarttheorem
      \mbox{}\par\nobreak\vskip\glueexpr-\parskip-\baselineskip+0.25em\relax\hrule\@height\z@
    \fi%
  \fi%
  \global\starttheoremfalse%
 \def\tempa{proof}%
 \ifx\tempa\mycurrenvir
    \ifstarttheorem
      \mbox{}\par\nobreak\vskip\glueexpr-\parskip-\baselineskip+0.25em\relax\hrule\@height\z@
    \fi%
 \fi%
 \global\starttheoremfalse%
}
\preto\enditemize{\global\starttheoremfalse}
\makeatother

\makeatletter
\preto\enumerate{%
  \if@inlabel
    \ifstarttheorem
      \mbox{}\par\nobreak\vskip\glueexpr-\parskip-\baselineskip+0.25em\relax\hrule\@height\z@
    \fi%
  \fi%
  \global\starttheoremfalse%
 \def\tempa{proof}%
 \ifx\tempa\mycurrenvir
    \ifstarttheorem
      \mbox{}\par\nobreak\vskip\glueexpr-\parskip-\baselineskip+0.25em\relax\hrule\@height\z@
    \fi%
 \fi%
 \global\starttheoremfalse%
}
\preto\endenumerate{\global\starttheoremfalse}
\makeatother

\newcommand{\tb}[2]{\substack{#1\\#2}}

\newcommand*{\psii}{{\psi^{-1}}}

\newcommand*{\psiis}[1]{{\psi_{#1}^{-1}}}
\newcommand*{\tpsi}{{\tilde{\psi}}}
\newcommand*{\tpsii}{{\tilde{\psi}^{-1}}}
\newcommand*{\tpsiis}[1]{{\tilde{\psi}_{#1}^{-1}}}
\newcommand*{\opsi}{{\mathring{\psi}}}
\newcommand*{\opsii}{{\mathring{\psi}^{-1}}}

\newcommand*{\IN}{\mathbb{N}}

\newcommand*{\IR}{\mathbb{R}}

\newcommand*{\Geo}{\operatorname{Geo}}

\newcommand*{\Sib}{\operatorname{Sib}}

\newcommand*{\Log}{\operatorname{Log}}
\newcommand*{\U}{\operatorname{U}}

\renewcommand*{\S}{\operatorname{S}}
\newcommand*{\tS}{\operatorname{\tilde{S}}}

\newcommand*{\I}{\mathbbm{1}}
\newcommand*{\rd}{\mathrm{d}}

\newcommand*{\D}{\operatorname{D}}

\newcommand*{\LS}{\mathcal{LS}}
\newcommand*{\LSi}{\LS^{-1}}

\renewcommand*{\P}{\mathbb{P}}

\newcommand*{\RV}{\operatorname{RV}}

\newcommand*{\R}{\textsf{R}}
\newcommand*{\eps}{\varepsilon}

\begin{document}
\thispagestyle{plain}
\begin{center}
  \sffamily
  {\bfseries\LARGE Right-truncated Archimedean and related copulas\par}
  \bigskip\smallskip
  {\Large Marius Hofert\footnote{Department of Statistics and Actuarial Science, University of
    Waterloo, 200 University Avenue West, Waterloo, ON, N2L
    3G1,
    \href{mailto:marius.hofert@uwaterloo.ca}{\nolinkurl{marius.hofert@uwaterloo.ca}}.
    Funding: The author would like to thank NSERC for financial support for this work
    through Discovery Grant RGPIN-5010-2015.}
    \par\bigskip
    \myisodate\par}
\end{center}
\par\smallskip
\begin{abstract}
  The copulas of random vectors with standard uniform univariate margins
  truncated from the right are considered and a general formula for such
  right-truncated conditional copulas is derived. This formula is analytical for
  copulas that can be inverted analytically as functions of each single
  argument. This is the case, for example, for Archimedean and related copulas.
  The resulting right-truncated Archimedean copulas are not only analytically
  tractable but can also be characterized as tilted Archimedean copulas. This
  finding allows one, for example, to more easily derive analytical properties
  such as the coefficients of tail dependence or sampling procedures of
  right-truncated Archimedean copulas. As another result, one can easily obtain
  a limiting Clayton copula for a general vector of truncation points converging
  to zero; this is an important property for (re)insurance and a fact already
  known in the special case of equal truncation points, but harder to prove without
  aforementioned characterization.  Furthermore, right-truncated Archimax
  copulas with logistic stable tail dependence functions are characterized as
  tilted outer power Archimedean copulas and an analytical form of
  right-truncated nested Archimedean copulas is also derived.
\end{abstract}
\minisec{Keywords}
Right truncation, conditional copulas, Archimedean copulas,
Archimax and nested Archimedean copulas, tilted and outer power transformations. %

\section{Introduction and motivation}
\cite{juriwuethrich2002} and \cite{charpentiersegers2007} studied the
practically relevant problem of determining the copula $C_t$ and its properties
of a bivariate random vector $(U_1,U_2)$ distributed according to some copula
$C$ given that, componentwise, $(U_1,U_2)\le (t,t)$ for some truncation point
$t\in(0,1]$. The copula $C_t$ is the copula of
$(U_1,U_2)\,|\,(U_1,U_2)\le (t,t)$, that is the copula of the conditional
distribution of $(U_1,U_2)$ given $(U_1,U_2)\le (t,t)$, a \emph{right-truncated}
$(U_1,U_2)$. %
In the first reference, $C_t$ is called ``extreme tail dependence copula
relative to $C$ at the level $t$'' %
as the limiting copula for $t\downarrow0$ is
of interest, and in the second reference $C_t$ is called ``lower tail dependence
copula relative to $C$ at level $t$''; note that $u$ instead of $t$ is used as single truncation point. %
Both references show that %
if $C$ is Archimedean, then $C_t$ is also Archimedean, and that if
$C$ is a Clayton copula, then $C_t$ equals $C$ -- a fact relevant for (re)insurance.
\cite{larssonneslehova2011} address the $d$-dimensional case,
but also focus on the limit for the truncation point converging to zero;
the corresponding copulas are referred to as ``limiting lower threshold copulas''. %

In comparison to the aforementioned publications, our contributions are as follows:
\begin{enumerate}
\item We consider a fixed $d$-dimensional vector
  $\bm{t}=(t_1,\dots,t_d)\in(0,1]^d$ as right truncation point. In particular,
  the thresholds do not have to be the same for each component %
  and we do not focus on the limiting case $\bm{t}\downarrow\bm{0}$ alone. %
\item We derive a formula for the copula $C_{\bm{t}}$ of
  $\bm{U}\,|\,\bm{U}\le\bm{t}$ in terms of the copula $C$ of
  $\bm{U}=(U_1,\dots,U_d)$; see Proposition~\ref{prop:rtrunc:C}. %
  The formula is analytical if $C$ is componentwise analytically invertible.
\item We consider the case where
  $\bm{U}$ follows an Archimedean copula and show that the family of copulas of
  $\bm{U}\,|\,\bm{U}\le\bm{t}$ is not only analytically available but actually
  also known, namely tilted Archimedean; see Theorem~\ref{thm:rtrunc:AC}.
\item We consider the case of
  $\bm{U}$ following outer power Archimedean copulas or Archimax copulas with
  logistic stable tail dependence function and show that the family of copulas
  of $\bm{U}\,|\,\bm{U}\le\bm{t}$ is tilted outer power Archimedean; see
  Sections~\ref{sec:rtrunc:opAC} and \ref{sec:rtrunc:AXC:log:stdf}.
\item We consider $\bm{U}$ following nested Archimedean copulas and
  derive the corresponding right-truncated copulas; see Theorem~\ref{thm:right:rtrunc:NAC}.
\end{enumerate}
Various examples are given and further properties %
discussed in the appendix.

The operation of right-truncation is important for (re)insurance as the copula
$C_{\bm{t}}$ allows one to study the dependence between the components of a
truncated loss random vector $\bm{L}\,|\,\bm{L}\le\bm{w}$ for any
$\bm{w}\in\IR^d$ such that $\P(\bm{L}\le\bm{w})>0$, where
$\bm{L} \sim F_{\bm{L}}$ with continuous margins $F_{L_1},\dots,F_{L_d}$ and
copula $C$.  To see this, note that the distribution function of
$\bm{L}\,|\,\bm{L}\le\bm{w}$ can be written as
$\P(\bm{L}\le \bm{x}\,|\,\bm{L}\le\bm{w})=\P(\bm{U}\le
\bm{u}\,|\,\bm{U}\le\bm{t})$ for
$\bm{U}=(F_{L_1}(L_1),\dots,F_{L_d}(L_d))\sim C$,
$\bm{u}=(F_{L_1}(x_1),\dots,F_{L_d}(x_d))$ and
$\bm{t}=(F_{L_1}(w_1),\dots,F_{L_d}(w_d))$, so in terms of the distribution
function of $\bm{X}=(\bm{U}\le \bm{u}\,|\,\bm{U}\le\bm{t})$ whose copula
$C_{\bm{t}}$ for a fixed truncation point $\bm{t}\in(0,1]^d$ is the main objective
in this work; occasionally, we also address the case $\bm{t}\downarrow\bm{0}$.
As the copula of a right-truncated distribution function, we simply refer to the
copula $C_{\bm{t}}$ of $\bm{U}\,|\,\bm{U}\le\bm{t}$ as \emph{right-truncated
  copula} in what follows.

Our findings also apply to the survival copula $\hat{C}$ of a left-truncated
$\bm{U}$, that is $\bm{U}\,|\,\bm{U}\ge\bm{t}$.
To see this let $\hat{\bm{u}}=\bm{1}-\bm{u}$, $\hat{\bm{t}}=\bm{1}-\bm{t}$ and
$\hat{\bm{U}}=\bm{1}-\bm{U}$, and note that, in distribution,
$(\bm{U} \ge \bm{u}\,|\,\bm{U} \ge \bm{t}) = (\bm{1}-\bm{U}\le
\bm{1}-\bm{u}\,|\,\bm{1}-\bm{U}\le\bm{1}-\bm{t}) = (\hat{\bm{U}} \le
\hat{\bm{u}} \,|\, \hat{\bm{U}} \le \hat{\bm{t}})$, so the survival copula of a
left-truncated $\bm{U}$ (that is the copula of the survival distribution of
$\bm{U}\sim C$ given $\bm{U}\ge \bm{t}$) equals the copula of $\hat{\bm{U}}$
(that is the survival copula $\hat{C}$ corresponding to $C$) right-truncated at
$\hat{\bm{t}}$. This fits in our framework if one considers the copulas we work
with (Archimedean, Archimax, etc.) as survival copulas $\hat{C}$.

\section{Right-truncated copulas}
We start with the general form of the distribution function $F_{\bm{t}}$ (and
its margins $F_{\bm{t},1},\dots,F_{\bm{t},d}$) of the right-truncated random
vector $\bm{U}\,|\,\bm{U}\le\bm{t}$ for $\bm{U}$ following a $d$-dimensional
copula $C$, that is the distribution function $F_{\bm{t}}$ of the conditional distribution of
$\bm{U}\sim C$ given $\bm{U}\le\bm{t}$.
\begin{lemma}[Right-truncated distribution function and its margins]\label{lem:rtrunc:df}
  Let $\bm{U}\sim C$ for a $d$-dimensional copula $C$ and let $\bm{t}\in(0,1]^d$ such that $C(\bm{t})>0$.
  Furthermore, let $\min\{\bm{x},\bm{t}\}=(\min\{x_1,t_1\},\dots,$ $\min\{x_d,t_d\})$.
  Then the distribution function $F_{\bm{t}}$ of $\bm{X}=(\bm{U}\,|\,\bm{U}\le\bm{t})$ is given
  by
  \begin{align*}
    F_{\bm{t}}(\bm{x})=\frac{C(\min\{\bm{x},\bm{t}\})}{C(\bm{t})}=\frac{C(\bm{x})}{C(\bm{t})},\quad\bm{x}\in[\bm{0},\bm{t}],
  \end{align*}
  with margins
  $F_{\bm{t},j}(x_j)=C(x_j;\bm{t}_{-j})/C(\bm{t})$,
  $x_j\in[0,t_j]$, $j=1,\dots,d$, where, for all $j=1,\dots,d$, $\bm{t}_{-j}=(t_1,\dots,t_{j-1},t_{j+1},\dots,t_d)$
  and
  \begin{align}
    C(x_j;\bm{t}_{-j})=C(t_1,\dots,t_{j-1},x_j,t_{j+1},\dots,t_d),\quad x_j\in[0,t_j].\label{comp:map}
  \end{align}
\end{lemma}
\begin{proof}
  The distribution function $F_{\bm{t}}$ of $\bm{U}\,|\,\bm{U}\le\bm{t}$ is given by
  \begin{align*}
    F_{\bm{t}}(\bm{x})&=\P(\bm{U}\le\bm{x}\,|\,\bm{U}\le \bm{t})=\frac{\P(\bm{U}\le\bm{x},\bm{U}\le \bm{t})}{\P(\bm{U}\le \bm{t})}=\frac{\P(\bm{U}\le\min\{\bm{x},\bm{t}\})}{\P(\bm{U}\le \bm{t})}
  \end{align*}
  which equals $\frac{\P(\bm{U}\le\bm{x})}{\P(\bm{U}\le\bm{t})}=\frac{C(\bm{x})}{C(\bm{t})}$ for $\bm{x}\in[\bm{0},\bm{t}]$.
  The $j$th margin is obtained by letting $x_{\tilde{j}}=t_{\tilde{j}}$ for all $\tilde{j}\neq j$.
\end{proof}
It is immediate from Lemma~\ref{lem:rtrunc:df} that if $C$ has density $c$, then $F_{\bm{t}}$
has density $f_{\bm{t}}(\bm{x})=c(\bm{x})/C(\bm{t})$, $\bm{x}\in(\bm{0},\bm{t})$.

The right-truncated distribution function $F_{\bm{t}}$ has a unique
copula %
which we denote by $C_{\bm{t}}$ and call \emph{right-truncated copula}.
\begin{definition}[Right-truncated copula]
  Let $\bm{U}\sim C$ for a $d$-dimensional copula $C$ and let $\bm{t}\in(0,1]^d$ such that $C(\bm{t})>0$.
  The copula
  $C_{\bm{t}}$ of the distribution function $F_{\bm{t}}$ of
  $\bm{U}\,|\,\bm{U}\le\bm{t}$ is called \emph{right-truncated copula at
    $\bm{t}$} or the \emph{copula $C$ right-truncated at
    $\bm{t}$}.
\end{definition}

A straightforward sampling procedure for $C_{\bm{t}}$ is the following.
\begin{algorithm}[Rejection sampling]\label{alg:crude:reject}
  \begin{enumerate}
  \item For $i=1,\dots,n$, do: Repeat sampling $\bm{U}\sim C$ until $\bm{U}\le\bm{t}$, then set $\bm{X}_i=\bm{U}$.
  \item Return
    $(F_{\bm{t},1}(X_{i1}),\dots,F_{\bm{t},d}(X_{id}))$, $i=1,\dots,n$. Alternatively, for
    sufficiently large $n$, return the pseudo-observations of
    $\bm{X}_1,\dots,\bm{X}_n$; see \cite{genestghoudirivest1995}.
  \end{enumerate}
\end{algorithm}

The following example shows pseudo-samples from right-truncated Marshall--Olkin copulas.
\begin{example}[Right-truncated bivariate Marshall--Olkin copulas]
  Figure~\ref{fig:rtrunc:MO:copula:samples} shows 5000
  pseudo-observations from bivariate right-truncated Marshall--Olkin copulas
  $C(u_1,u_2)=\min\{u_1^{1-\alpha_1}u_2,\ u_1u_2^{1-\alpha_2}\}$ with parameters
  $\alpha_1=0.2$, $\alpha_2=0.7$ and truncation points as indicated; for
  $\bm{t}=(1,1)$ in the top left plot, no truncation takes place and thus a
  sample from $C$ is shown.
  \begin{figure}[htbp]
    \centering
    \includegraphics[width=0.48\textwidth]{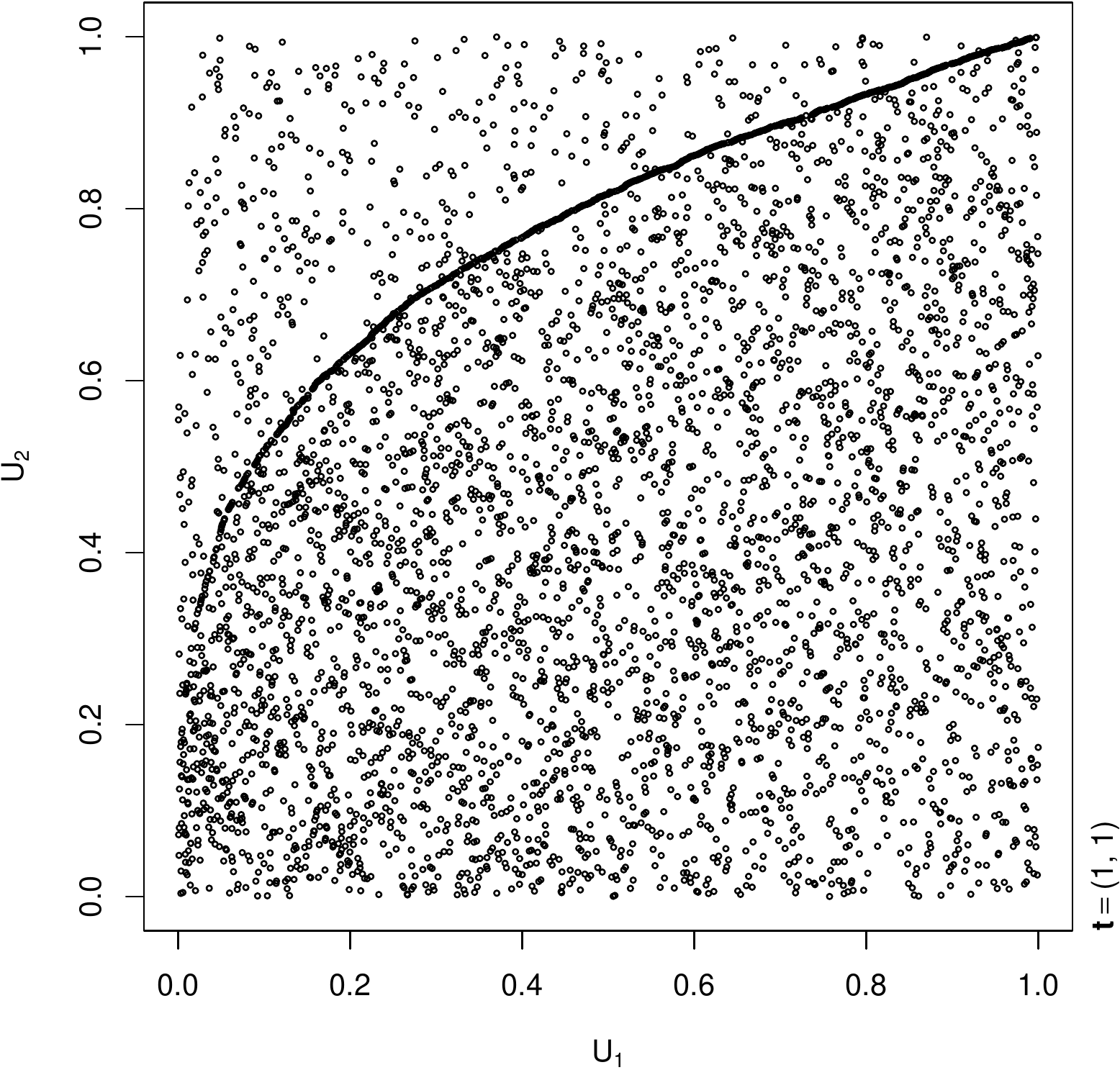}%
    \hfill
    \includegraphics[width=0.48\textwidth]{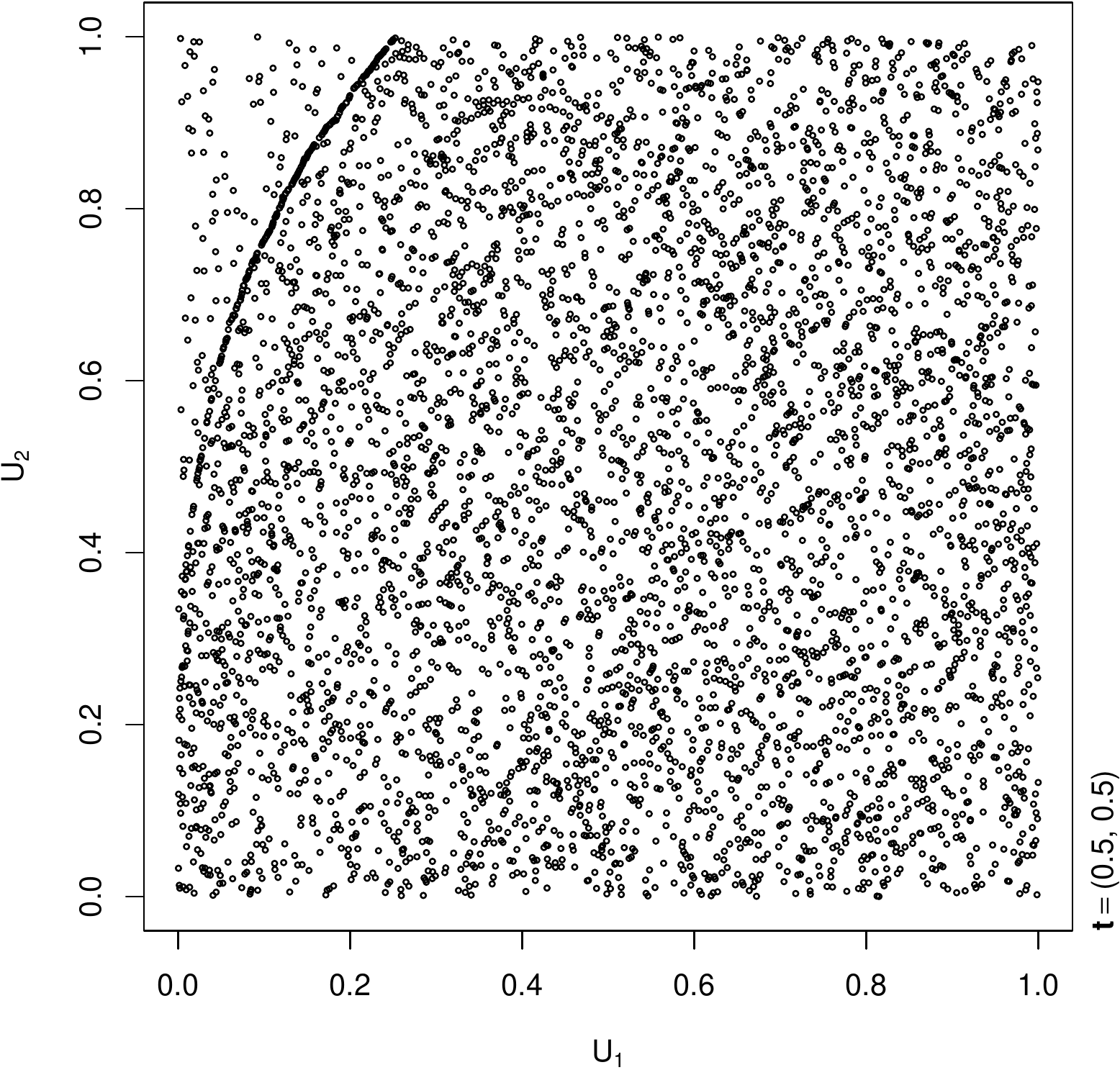}\\[2mm]
    \includegraphics[width=0.48\textwidth]{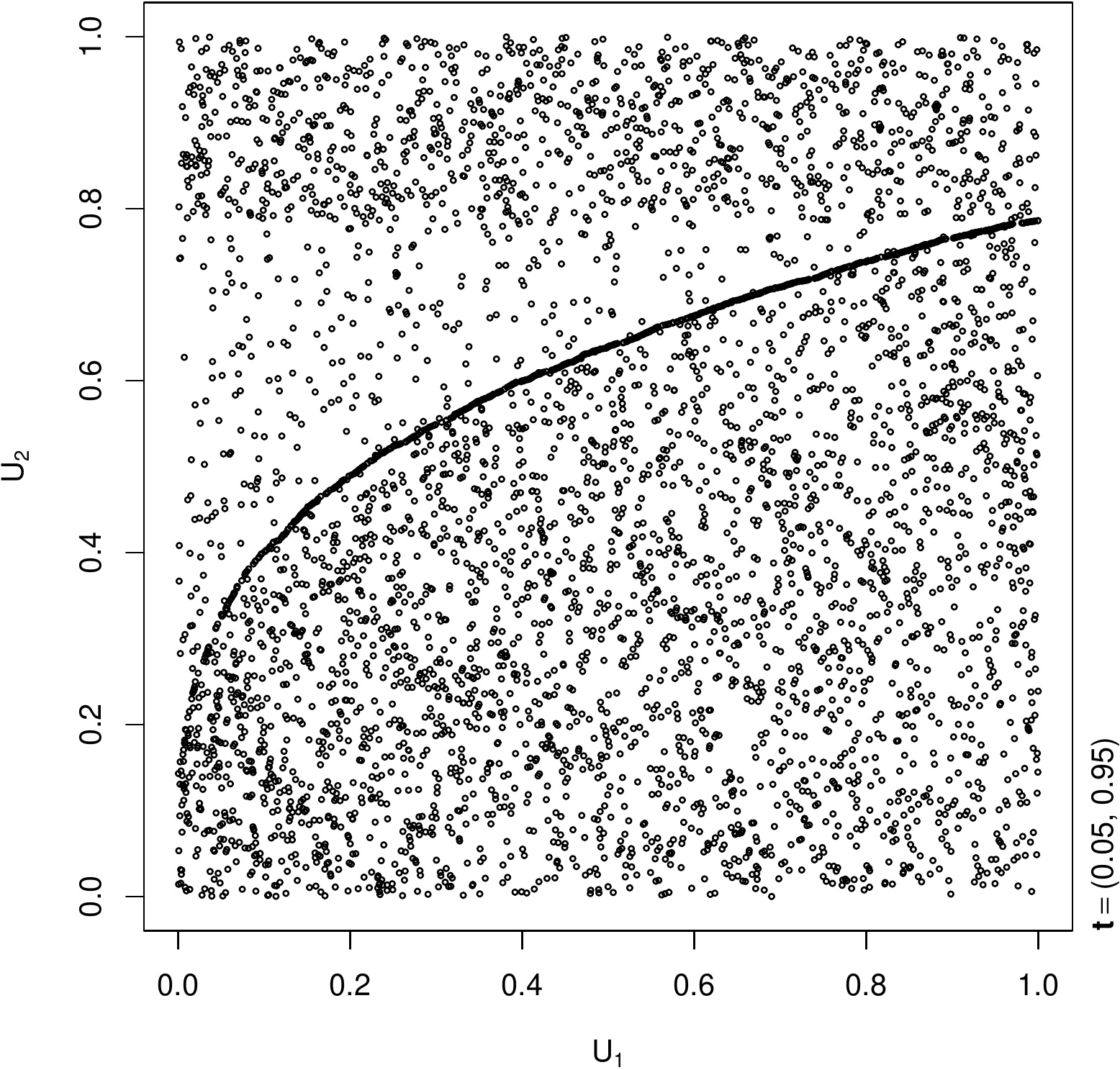}%
    \hfill
    \includegraphics[width=0.48\textwidth]{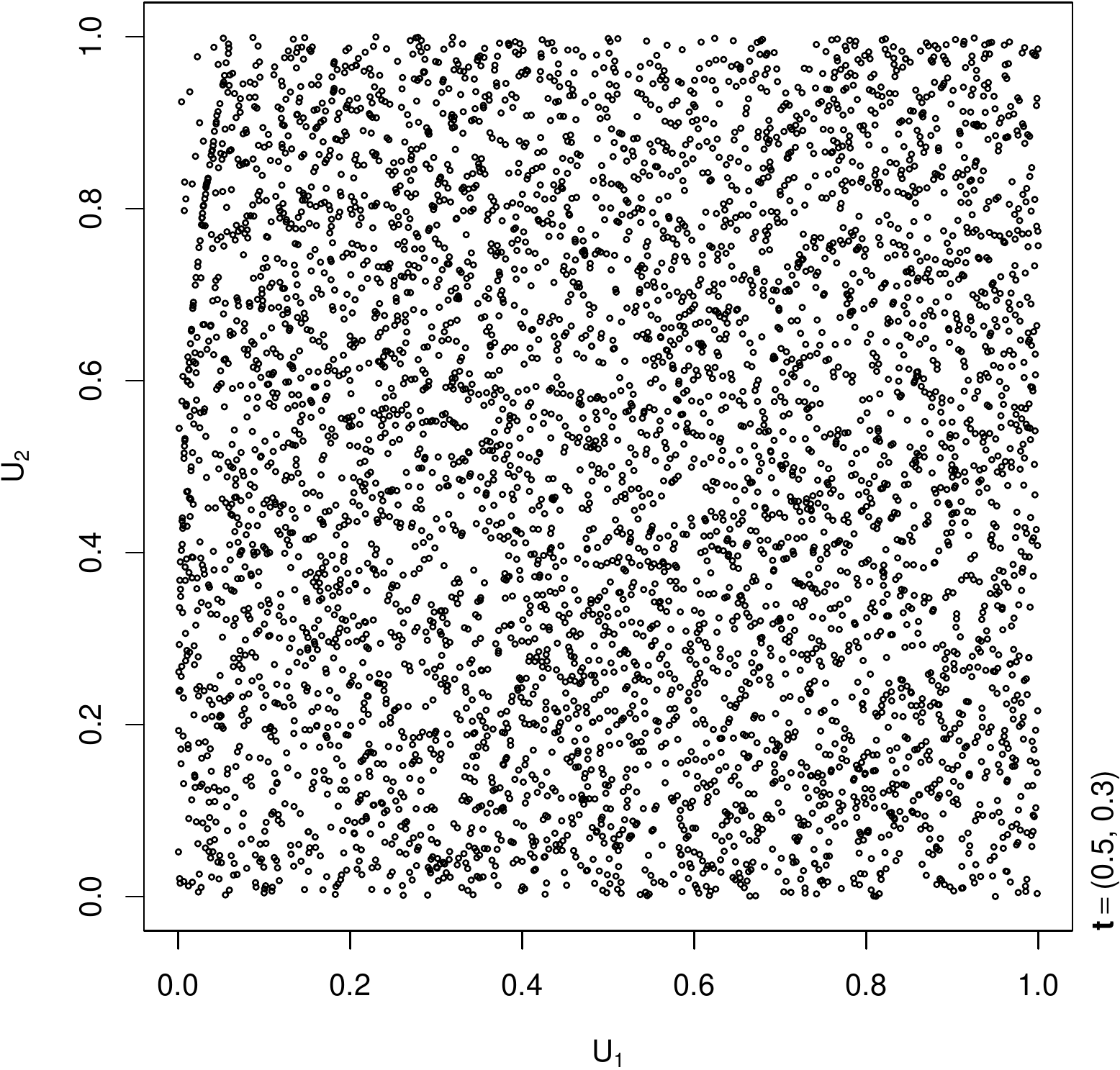}%
    \caption{$n=5000$ pseudo-observations from a Marshall--Olkin copula with
      parameters $\alpha_1=0.2$, $\alpha_2=0.7$ right-truncated at the indicated
      points $\bm{t}=(t_1,t_2)$.}
    \label{fig:rtrunc:MO:copula:samples}
  \end{figure}
  We see how right-truncation allows one to change the shape of Marshall--Olkin
  samples. In particular, right-truncation allows one to cut out a lower-left
  region of the copula samples (appropriately scaled to again be copula samples
  after truncation) and thus to shift the top right end point of the singular
  component to points other than $(1,1)$. Although we consider Archimedean and
  related copulas in what follows, this neither symmetric nor radially symmetric
  example nicely demonstrates the operation of right-truncation.
\end{example}

We now derive a general formula for the right-truncated copula $C_{\bm{t}}$ of a
$d$-dimensional copula $C$. To this end and for later, it will be convenient to
write $C(\{u_j\}_j)$ for $C(u_1,\dots,u_d)$ and to let
$y_j\mapsto C^{-1}[y_j;\bm{t}_{-j}]$ denote the generalized inverse of the
increasing (that is non-decreasing) function $x_j\mapsto C(x_j;\bm{t}_{-j})$;
see \cite{embrechtshofert2013c} for the notion of generalized inverses.
\begin{proposition}[Right-truncated copulas]\label{prop:rtrunc:C}
  Let $C$ be a $d$-dimensional copula and let $\bm{t}\in(0,1]^d$ such that $C(\bm{t})>0$.
  Then the right-truncated copula at $\bm{t}$ is given by
  \begin{align*}
    C_{\bm{t}}(\bm{u})=\frac{C(\{C^{-1}[C(\bm{t})u_j;\bm{t}_{-j}]\}_j)}{C(\bm{t})},\quad\bm{u}\in[0,1]^d.
  \end{align*}
\end{proposition}
\begin{proof}
  By Lemma~\ref{lem:rtrunc:df}, the quantile function of $F_{\bm{t},j}$ is $F_{\bm{t},j}^{-1}(u_j)=C^{-1}[C(\bm{t})u_j;\bm{t}_{-j}]$.
  By Sklar's Theorem, the
  right-truncated copula can thus be obtained from $F_{\bm{t}}$ via
  \begin{align*}
    C_{\bm{t}}(\bm{u})&=F_{\bm{t}}(F_{\bm{t},1}^{-1}(u_1),\dots,F_{\bm{t},d}^{-1}(u_d))=\frac{C(F_{\bm{t},1}^{-1}(u_1),\dots,F_{\bm{t},d}^{-1}(u_d))}{C(\bm{t})}\\
    &=\frac{C(C^{-1}[C(\bm{t})u_1;\bm{t}_{-1}],\dots,C^{-1}[C(\bm{t})u_d;\bm{t}_{-d}])}{C(\bm{t})}=\frac{C(\{C^{-1}[C(\bm{t})u_j;\bm{t}_{-j}]\}_j)}{C(\bm{t})}.\qedhere
  \end{align*}
\end{proof}

\begin{remark}
  If $U\sim\U(0,1)$ and $t\in(0,1]$, then $X=U\,|\,U\le t$ has distribution
  function $F_t(x)=\P(U\le x\,|\,U\le t)=x/t$, $x\in[0,t]$, which equals the
  distribution function of the random variable $Y=tU$. As such, one might be tempted to believe that
  in the multivariate case for $\bm{U}\sim C$, the random vector
  $\bm{U}\,|\,\bm{U}\le\bm{t}$ is in distribution equal to (and thus can be
  sampled as) $\bm{t}\bm{U}=(t_1U_1,\dots,t_dU_d)$. However, note that
  $\bm{t}\bm{U}$ is a simple componentwise scaled version of $\bm{U}\sim C$ and thus,
  by the invariance principle, the copula of $\bm{t}\bm{U}$ is also $C$. %
\end{remark}

\begin{example}[Independence, independent blocks and comonotonicity copulas]\label{ex:indep:dep:blocks:max}
  If $C$ is \emph{componentwise analytically invertible}, that is componentwise
  invertible in analytic form, we see from Proposition~\ref{prop:rtrunc:C} that
  the right-truncated copula $C_{\bm{t}}$ is given analytically. The following are immediate examples.
  \begin{enumerate}
  \item For the independence copula $C(\bm{u})=\prod_{j=1}^d u_j$, we have
    $C(x_j;\bm{t}_{-j})=x_j\prod_{\tilde{j}\neq j}t_{\tilde{j}}$,
    $x_j\in[0,t_j]$, so that
    $C^{-1}[y_j;\bm{t}_{-j}]=y_j/\prod_{\tilde{j}\neq j}t_{\tilde{j}}$,
    $y_j\in[0,\ C(t_j;\bm{t}_{-j})=C(\bm{t})]$, and thus
    \begin{align*}
      C_{\bm{t}}(\bm{u})&=\frac{C(\{C^{-1}[C(\bm{t})u_j;\bm{t}_{-j}]\}_j)}{C(\bm{t})}=\frac{\prod_{j=1}^d\frac{C(\bm{t})u_j}{\prod_{\tilde{j}\neq j}t_{\tilde{j}}}}{C(\bm{t})}=C(\bm{t})^{d-1}\prod_{j=1}^d \frac{u_j}{\prod_{\tilde{j}\neq j}t_{\tilde{j}}}\\
                        &=C(\bm{t})^{d-1}\frac{\prod_{j=1}^d u_j}{\prod_{j=1}^d\prod_{\tilde{j}\neq j}t_{\tilde{j}}}=C(\bm{t})^{d-1}\frac{C(\bm{u})}{C(\bm{t})^{d-1}}=C(\bm{u}),
    \end{align*}
    which confirms the intuition that right-truncating independent random variables (the independence copula) leads to
    independent random variables (the independence copula).
  \item\label{ex:indep:dep:blocks:max:indep:blocks} For a hierarchical copula of the form
    $C(\bm{u})=\prod_{s=1}^SC_s(\bm{u}_s)$ (independent dependent blocks) with
    $\bm{u}=(\bm{u}_1,\dots,\bm{u}_S)$, copulas $C_s$, $s=1,\dots,S$, and a
    truncation point $\bm{t}=(\bm{t}_1,\dots,\bm{t}_S)$ we have
    $C(x_{sj};\bm{t}_{-sj})=C_s(x_{sj};\bm{t}_{s(-j)})\prod_{\tb{\tilde{s}=1}{\tilde{s}\neq s}}^S C_{\tilde{s}}(\bm{t}_{\tilde{s}})$, where $\bm{t}_{-sj}$
    denotes $\bm{t}$ without the $j$th component in block $s$
    and $\bm{t}_{s(-j)}$ denotes $\bm{t}_s$ without the $j$th component.
    Therefore $C^{-1}[y_{sj};\bm{t}_{-sj}]=C_s^{-1}\bigl(y_{sj}/\prod_{\tb{\tilde{s}=1}{\tilde{s}\neq s}}^S C_{\tilde{s}}(\bm{t}_{\tilde{s}}); \bm{t}_{s(-j)}\bigr)$ and thus
    \begin{align*}
      C_{\bm{t}}(\bm{u})&=\frac{C\bigl(\{C_s^{-1}(C(\bm{t})u_{sj}/\prod_{\tb{\tilde{s}=1}{\tilde{s}\neq s}}^S C_{\tilde{s}}(\bm{t}_{\tilde{s}}); \bm{t}_{s(-j)})\}_{s,j}\bigr)}{C(\bm{t})}\\
                        &=\frac{C(\{C_s^{-1}(C_s(\bm{t}_s)u_{sj}; \bm{t}_{s(-j)})\}_{s,j})}{C(\bm{t})}=\prod_{s=1}^S\frac{C_s(\{C_s^{-1}(C_s(\bm{t}_s)u_{sj}; \bm{t}_{s(-j)})\}_{s,j})}{C_s(\bm{t}_s)}\\
                        &=\prod_{s=1}^SC_{s,\bm{t}_s}(\bm{u}_s),
    \end{align*}
    that is the product of the copulas $C_1,\dots,C_S$ right-truncated at $\bm{t}_1,\dots,\bm{t}_S$, respectively.
    This confirms the intuition that right-truncating independent random vectors (their copulas)
    leads to independent right-truncated random vectors (their copulas).
  \item For the comonotonicity copula $C(\bm{u})=\min\{\bm{u}\}=\min\{u_1,\dots,u_d\}$, we have
    \begin{align*}
      C(x_j;\bm{t}_{-j})=\min\{t_1,\dots,t_{j-1},x_j,t_{j+1},\dots,x_d\}=
      \begin{cases}
        x_j,& x_j\le\min\{\bm{t}_{-j}\},\\
        \min\{\bm{t}_{-j}\},& \min\{\bm{t}_{-j}\}<x_j\le t_j,
      \end{cases}
    \end{align*}
    where the second case is void %
    if $t_j\le \min\{\bm{t}_{-j}\}$. Therefore, %
    $C^{-1}[y_j;\bm{t}_{-j}]=y_j$, $y_j\in[0,\ C(t_j;\bm{t}_{-j})=C(\bm{t})]$, and thus
    \begin{align*}
      C_{\bm{t}}(\bm{u})&=\frac{C(\{C^{-1}[C(\bm{t})u_j;\bm{t}_{-j}]\}_j)}{C(\bm{t})}=\frac{\min\{\{C(\bm{t})u_j\}_j\}}{C(\bm{t})}=\frac{C(\bm{t})\min\{\{u_j\}_j\}}{C(\bm{t})}=C(\bm{u}),
    \end{align*}
    which confirms the intuition that right-truncating comontone random
    variables (the comonotonicity copula) %
    leads to comontone random variables (the comonotonicity copula).
  \end{enumerate}
\end{example}

\begin{example}[Bivariate right-truncated Marshall--Olkin copulas]\label{ex:biv:MO}
  Another example of a componentwise analytically invertible copula is the
  bivariate Marshall--Olkin copula
  $C(u_1,u_2)=\min\{u_1^{1-\alpha_1}u_2,\ u_1u_2^{1-\alpha_2}\}$ for
  $\alpha_1,\alpha_2\in(0,1)$. %
  Note that $u_1^{1-\alpha_1}u_2\lesseqgtr u_1u_2^{1-\alpha_2}$ if and only if
  $u_2^{\alpha_2}\lesseqgtr u_1^{\alpha_1}$ if and only if
  $u_2\lesseqgtr u_1^{\alpha_1/\alpha_2}$ if and only if
  $u_1\gtreqless u_2^{\alpha_2/\alpha_1}$. Let $\alpha_{-j}=\alpha_1$ if $j=2$ and $\alpha_{-j}=\alpha_2$ if $j=1$.
  Then
  \begin{align*}
    C(x_j;t_{-j})=\begin{cases}%
      t_{-j}^{1-\alpha_{-j}}x_j,& 0\le x_j< t_{-j}^{\alpha_{-j}/\alpha_{j}},\\
      t_{-j}x_j^{1-\alpha_j},& t_{-j}^{\alpha_{-j}/\alpha_j}\le x_j\le 1,
    \end{cases}
  \end{align*}
  which is continuous %
  and strictly increasing, and equal to $t_{-j}$ for $x_j=1$.
  Therefore,
  \begin{align*}
    C^{-1}[y_j;t_{-j}]=\begin{cases}%
      y_j/t_{-j}^{1-\alpha_{-j}},& 0\le y_j\le t_{-j}^{1-\alpha_{-j}+\alpha_{-j}/\alpha_j},\\
      (y_j/t_{-j})^{\frac{1}{1-\alpha_j}},& t_{-j}^{1-\alpha_{-j}+\alpha_{-j}/\alpha_j}< y_j\le t_{-j};
    \end{cases}
  \end{align*}
  note that since $\alpha_j\le 1$, we have $1-\alpha_{-j}+\alpha_{-j}/\alpha_j\ge 1$ so that $t_{-j}^{1-\alpha_{-j}+\alpha_{-j}/\alpha_j} \le t_{-j}$.

  \emph{Case 1.} If $t_2^{\alpha_2}\le t_1^{\alpha_1}$, we have $C(\bm{t})=t_1^{1-\alpha_1}t_2$ and a calculation shows that
  \begin{align*}
    C^{-1}[C(\bm{t})u_1;t_2]
    &=\begin{cases}
      t_1^{1-\alpha_1}t_2^{\alpha_2} u_1,& 0\le u_1\le (t_{2}^{\alpha_{2}}/t_1^{\alpha_1})^{(1-\alpha_1)/\alpha_1},\\
      t_1u_1^{\frac{1}{1-\alpha_1}},& (t_{2}^{\alpha_{2}}/t_1^{\alpha_1})^{(1-\alpha_1)/\alpha_1} < u_1\le 1, %
    \end{cases}%
  \end{align*}
  and that $C^{-1}[C(\bm{t})u_2;t_1]=t_2u_2$, $0\le u_2\le 1$.
  We thus obtain that
  \begin{align*}
    C_{\bm{t}}(u_1,u_2)%
                       &= \frac{\min\{C^{-1}[C(\bm{t})u_1;t_2]^{1-\alpha_1}C^{-1}[C(\bm{t})u_2;t_1],\ C^{-1}[C(\bm{t})u_1;t_2]C^{-1}[C(\bm{t})u_2;t_1]^{1-\alpha_2}\}}{C(\bm{t})}\\
                       &=\begin{cases}
                         \min\bigl\{(t_2^{\alpha_2}/t_1^{\alpha_1})^{1-\alpha_1}u_1^{1-\alpha_1}u_2,\ u_1u_2^{1-\alpha_2}\bigr\},& 0\le u_1\le (t_{2}^{\alpha_{2}}/t_1^{\alpha_1})^{(1-\alpha_1)/\alpha_1},\\
                         \min\bigl\{u_1u_2,\ (t_1^{\alpha_1}/t_2^{\alpha_2})u_1^{\frac{1}{1-\alpha_1}}u_2^{1-\alpha_2}\bigr\},& (t_{2}^{\alpha_{2}}/t_1^{\alpha_1})^{(1-\alpha_1)/\alpha_1} < u_1\le 1.
                       \end{cases}
  \end{align*}
  Furthermore, the singular component of $C_{\bm{t}}(u_1,u_2)$ is given by all
  $(u_1,u_2)\in[0,1]^2$ such that
  $C^{-1}[C(\bm{t})u_1;t_{2})]^{\alpha_1}=C^{-1}[C(\bm{t})u_2;t_{1})]^{\alpha_2}$,
  which are all $u_2$ such that
  \begin{align*}
    u_2^{\alpha_2}
    &=(t_1^{\alpha_1}/t_2^{\alpha_2})^{1-\alpha_1}u_1^{\alpha_1},\quad 0\le u_1\le (t_{2}^{\alpha_{2}}/t_1^{\alpha_1})^{\frac{1-\alpha_1}{\alpha_1}}.
  \end{align*}

  \emph{Case 2.} If $t_2^{\alpha_2}>t_1^{\alpha_1}$, we have $C(\bm{t})=t_1t_2^{1-\alpha_2}$ and a calculation shows that $C^{-1}[C(\bm{t})u_1;t_2]=t_1u_1$, $0\le u_1\le 1$,
  and that
  \begin{align*}
    C^{-1}[C(\bm{t})u_2;t_1]
    &=\begin{cases}
      t_1^{\alpha_1}t_2^{1-\alpha_2}u_2,& 0\le u_2\le (t_1^{\alpha_1}/t_2^{\alpha_2})^{(1-\alpha_2)/\alpha_2},\\
      t_2u_2^{\frac{1}{1-\alpha_2}},& (t_1^{\alpha_1}/t_2^{\alpha_2})^{(1-\alpha_2)/\alpha_2}< u_2\le 1.
    \end{cases}
  \end{align*}
  We thus obtain that
  \begin{align*}
    C_{\bm{t}}(u_1,u_2)
    &=\begin{cases}
      \min\bigl\{u_1^{1-\alpha_1}u_2,\ (t_1^{\alpha_1}/t_2^{\alpha_2})^{1-\alpha_2}u_1u_2^{1-\alpha_2}\bigr\},& 0\le u_2\le (t_1^{\alpha_1}/t_2^{\alpha_2})^{(1-\alpha_2)/\alpha_2},\\
      \min\bigl\{(t_2^{\alpha_2}/t_1^{\alpha_1})u_1^{1-\alpha_1}u_2^{\frac{1}{1-\alpha_2}},\ u_1u_2\bigr\},& (t_1^{\alpha_1}/t_2^{\alpha_2})^{(1-\alpha_2)/\alpha_2} < u_2\le 1,
    \end{cases}
  \end{align*}
  which can also be obtained from the case $t_2^{\alpha_2}\le t_1^{\alpha_1}$ by interchanging $t_1,t_2$ and $\alpha_1,\alpha_2$ and $u_1,u_2$. %
  Furthermore, the singular component is given
  by all $u_2$ such that
  \begin{align*}
    u_2^{\alpha_2}
    &=(t_1^{\alpha_1}/t_2^{\alpha_2})^{1-\alpha_2}u_1^{\alpha_1},\quad 0\le u_1\le 1.
  \end{align*}

  Figure~\ref{fig:rtrunc:MO:copula:copulas} shows the bivariate right-truncated Marshall--Olkin copulas
  corresponding to the samples displayed in Figure~\ref{fig:rtrunc:MO:copula:samples}; with singular components
  depicted on the graphs of $C_{\bm{t}}$.

  One can also use above formulas for $C_{\bm{t}}$ to verify that if $t_1=t_2=t$, then
  $\lim_{t\downarrow 0}C_{(t,t)}(u_1,u_2)$ $=u_1u_2$ if $\alpha_1\neq\alpha_2$ and $\lim_{t\downarrow 0}C_{(t,t)}(u_1,u_2)=\min\{u_1^{1-\alpha}u_2,\ u_1u_2^{1-\alpha}\}$
  if $\alpha_1=\alpha_2=\alpha$, so limiting copulas of right-truncated Marshall--Olkin
  copulas are the independence copula or Cuadras--Aug\'e copulas.
  \begin{figure}[htbp]
    \centering
    \includegraphics[width=0.48\textwidth]{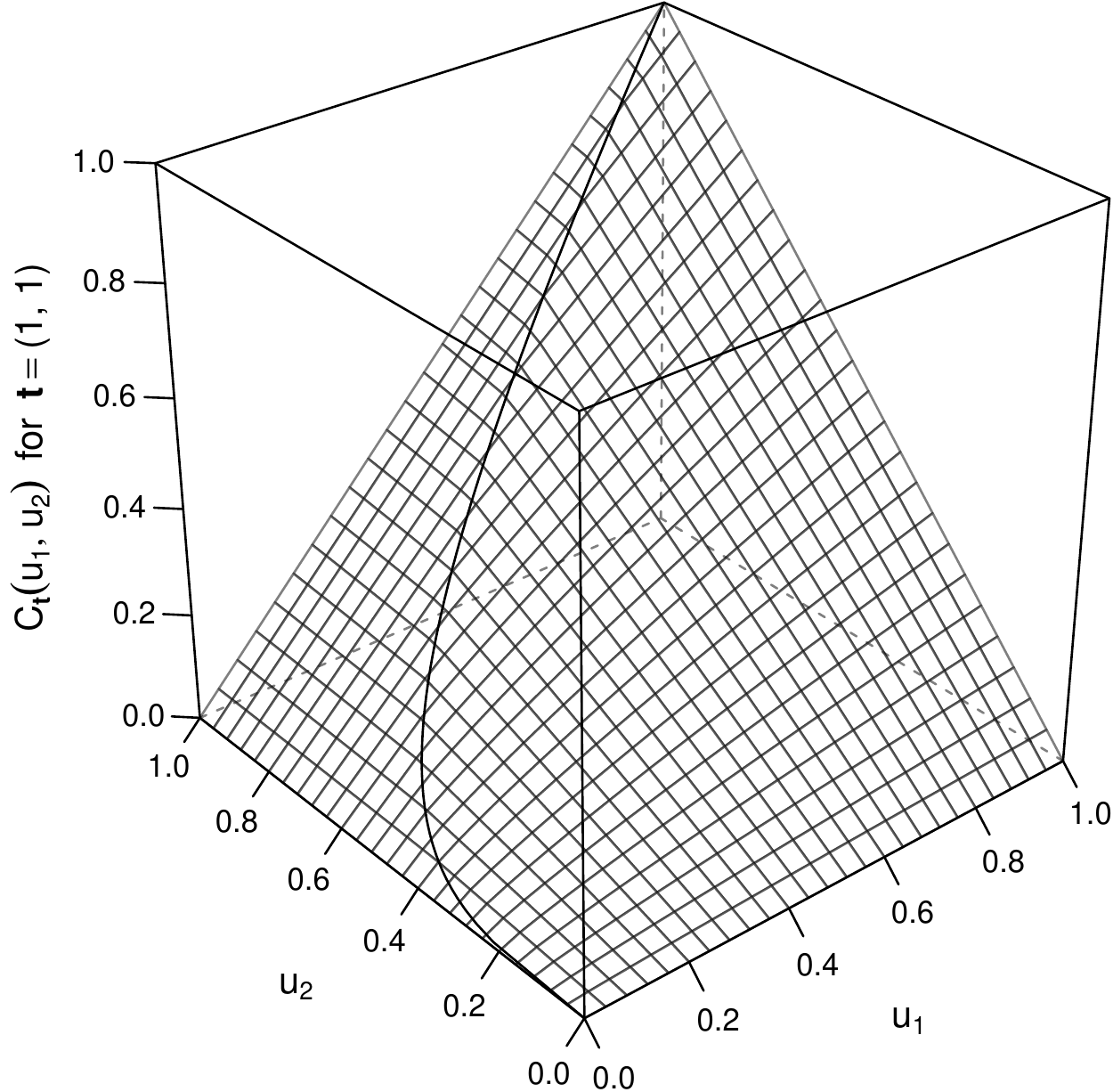}%
    \hfill
    \includegraphics[width=0.48\textwidth]{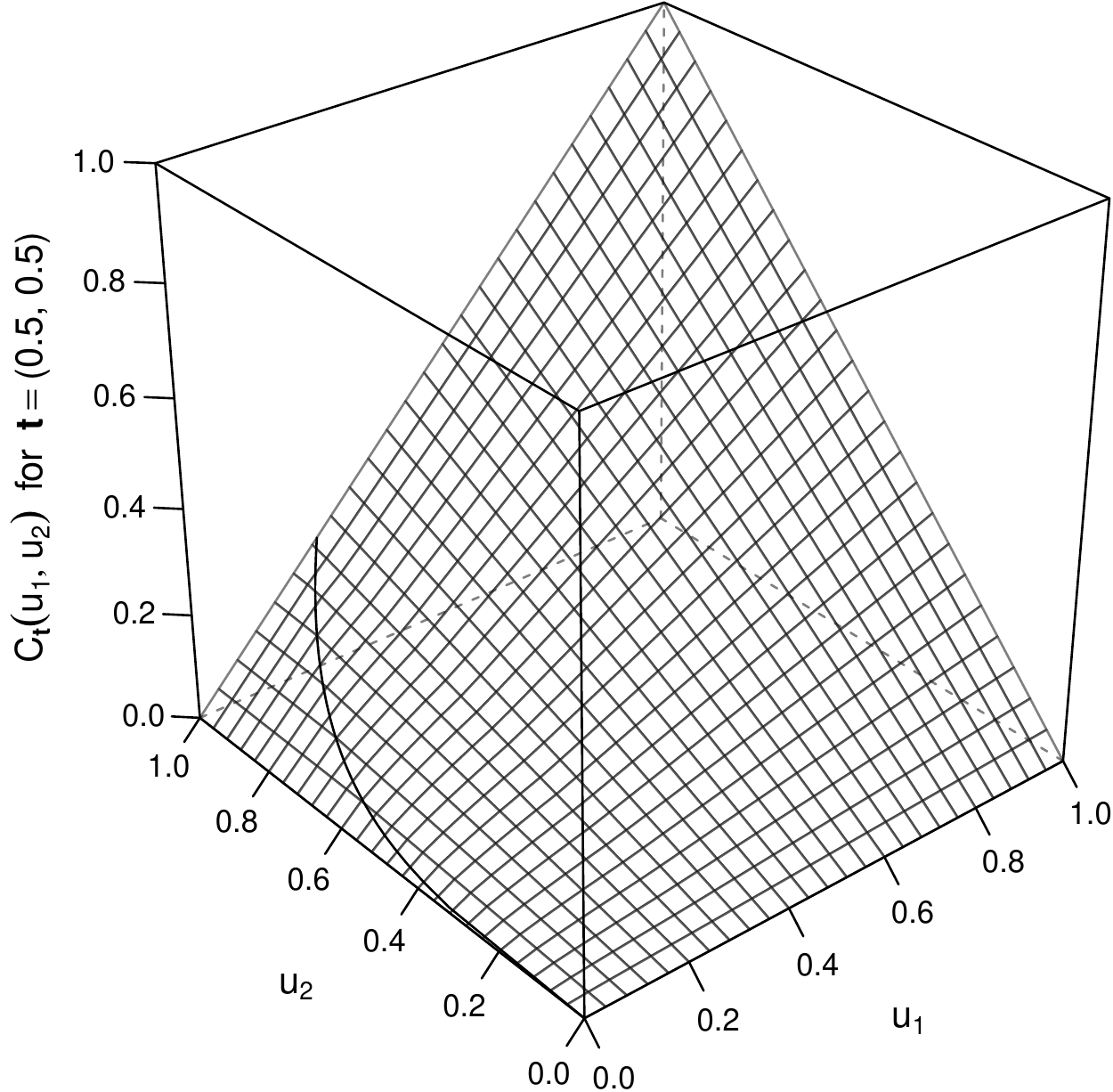}\\[2mm]
    \includegraphics[width=0.48\textwidth]{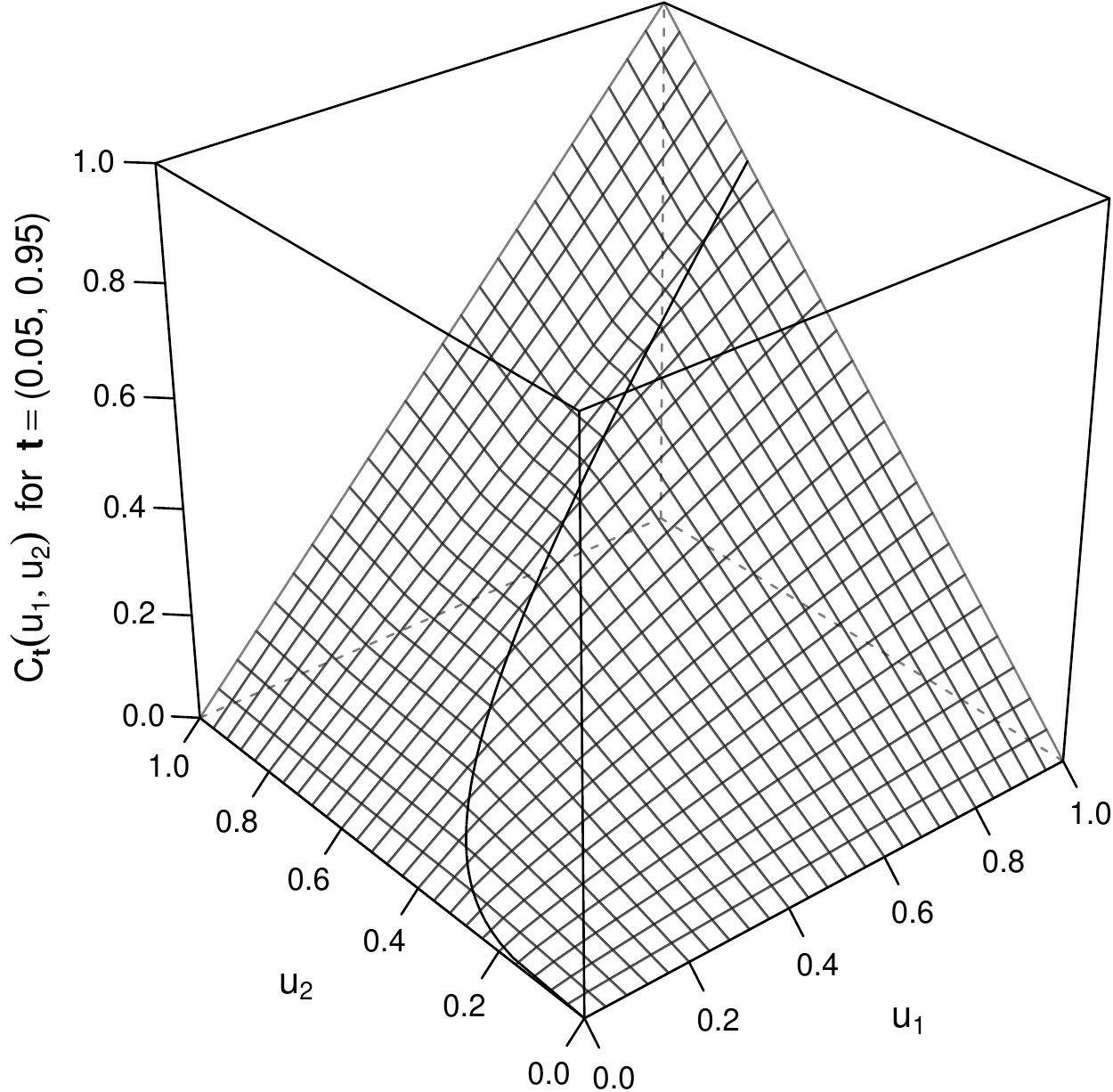}%
    \hfill
    \includegraphics[width=0.48\textwidth]{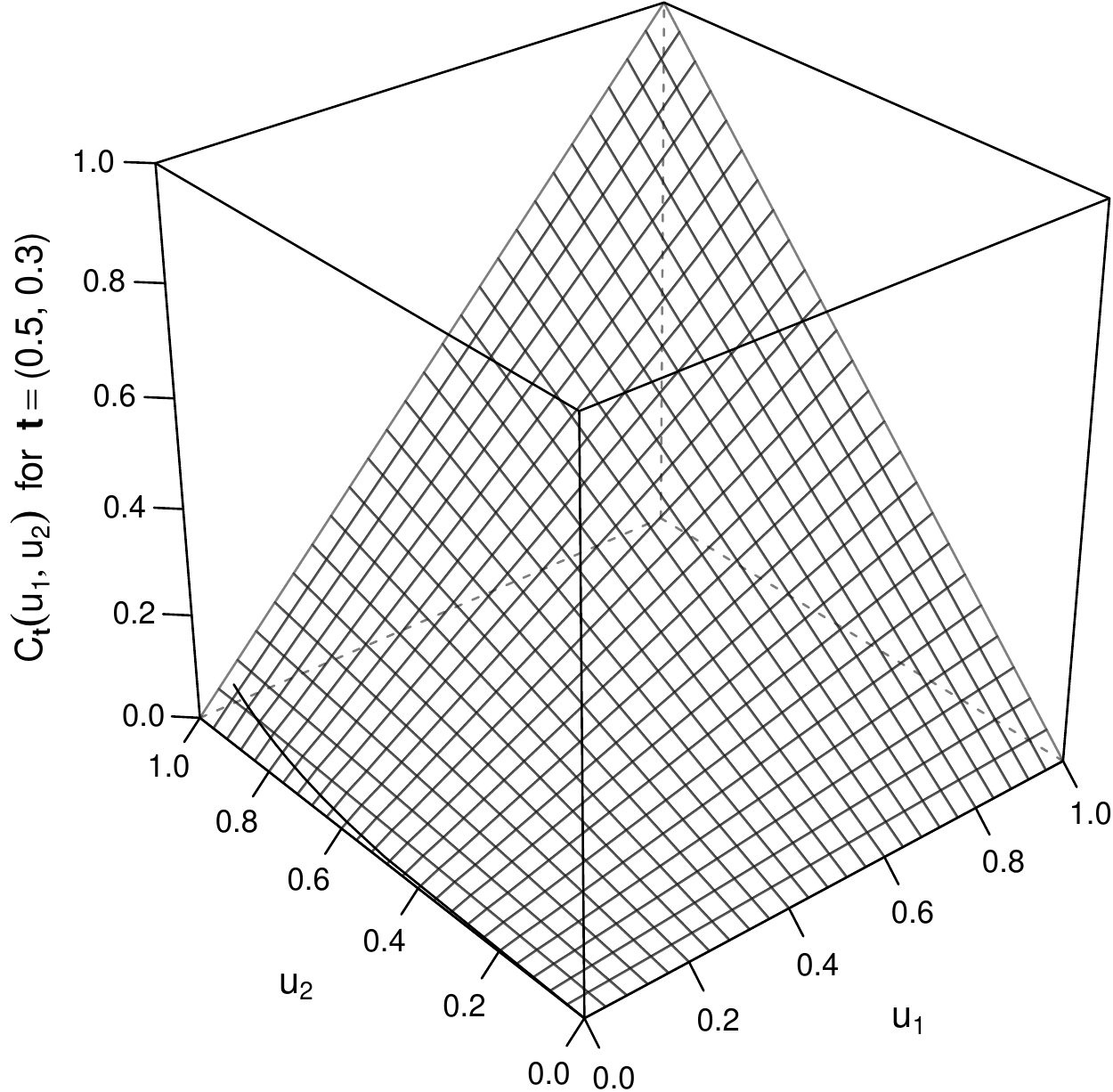}%
    \caption{Bivariate Marshall--Olkin copulas with
      parameters $\alpha_1=0.2$, $\alpha_2=0.7$ right-truncated at the indicated
      points $\bm{t}$.}
    \label{fig:rtrunc:MO:copula:copulas}
  \end{figure}
\end{example}

The following result provides a scaling property of right-truncated extreme
value copulas.  It implies that unless no truncation takes place (so unless
$\bm{t}=\bm{1}$), right-truncated extreme value copulas are not extreme value
copulas anymore. In particular, the bivariate right-truncated Marshall--Olkin
copulas from Example~\ref{ex:biv:MO} are not extreme value anymore.
\begin{proposition}[Scaling of extreme value copulas]\label{prop:scaling:EVCs}
  If $C$ is an extreme value copula, then $C_{\bm{t}}(\bm{u}^{\alpha})=C_{\bm{t}^{1/\alpha}}^\alpha(\bm{u})$, $\bm{u}\in[0,1]^d$,
  for all $\alpha>0$.
\end{proposition}
\begin{proof}
  Note that $C^{-1}[y_j^\alpha;\bm{t}_{-j}]=v$ if and only if $C(t_1,\dots,t_{j-1},v,t_{j+1},\dots,t_d)=y_j^{\alpha}$ if and only if
  $C^{1/\alpha}(t_1,\dots,t_{j-1},v,t_{j+1},\dots,t_d)=y_j$ if and only if $C(t_1^{1/\alpha},\dots,t_{j-1}^{1/\alpha},v^{1/\alpha},t_{j+1}^{1/\alpha},\dots,$ $t_d^{1/\alpha})=y_j$
  if and only if $C^{-1}[y_j;\bm{t}_{-j}^{1/\alpha}]=v^{1/\alpha}$ if and only if $C^{-1}[y_j;\bm{t}_{-j}^{1/\alpha}]^\alpha=v$.
  Therefore, $C^{-1}[y_j^\alpha;\bm{t}_{-j}]=C^{-1}[y_j;\bm{t}_{-j}^{1/\alpha}]^\alpha$ and thus
  \begin{align*}
    C_{\bm{t}}(\bm{u}^\alpha)&=\frac{C(\{C^{-1}[C(\bm{t})u_j^\alpha;\bm{t}_{-j}]\}_j)}{C(\bm{t})}=\frac{C(\{C^{-1}[(C(\bm{t}^{1/\alpha})u_j)^\alpha;\bm{t}_{-j}]\}_j)}{C(\bm{t})}\\
                             &=\frac{C(\{C^{-1}[C(\bm{t}^{1/\alpha})u_j;\bm{t}_{-j}^{1/\alpha}]^\alpha\}_j)}{C(\bm{t})}=\frac{C^\alpha(\{C^{-1}[C(\bm{t}^{1/\alpha})u_j;\bm{t}_{-j}^{1/\alpha}]\}_j)}{C(\bm{t})}\\
    &=\biggl(\frac{C(\{C^{-1}[C(\bm{t}^{1/\alpha})u_j;\bm{t}_{-j}^{1/\alpha}]\}_j)}{C(\bm{t}^{1/\alpha})}\biggr)^\alpha=C_{\bm{t}^{1/\alpha}}^\alpha(\bm{u}),\quad \bm{u}\in[0,1]^d.\qedhere
  \end{align*}
\end{proof}

Selected further properties of general right-truncated copulas are derived in Appendix~\ref{sec:prop}.

\section{Right-truncated Archimedean copulas}%
In this section we characterize right-truncated Archimedean copulas as tilted Archimedean copulas,
address their properties and consider examples of right-truncated Archimedean copulas.

\subsection{Characterization and properties}
Archimedean copulas are widely used in finance, insurance and
risk management. A copula $C$ is an \emph{Archimedean copula} if it admits the
form
\begin{align}
  C(\bm{u})=\psi(\psii[u_1]+\dots+\psii[u_d]),\quad \bm{u}=(u_1,\dots,u_d)\in[0,1]^d,\label{eq:form:AC}
\end{align}
where $\psi:[0,\infty)\to[0,1]$ is known as \emph{(Archimedean) generator}.  Let
$\Psi_d$ denote the set of all Archimedean generators which generate a
$d$-dimensional Archimedean copula; see, for example,
\cite{mcneilneslehova2009}.

It is immediate from \eqref{eq:form:AC} that Archimedean copulas $C$ are
componentwise analytically invertible and thus lead to analytical
right-truncated copulas $C_{\bm{t}}$. The following result not only shows that
right-truncated Archimedean copulas are Archimedean again (known in the
bivariate case and for equal truncation points; see \cite{juriwuethrich2002} and
\cite{charpentiersegers2007}), but also that we know their type, even in the
general $d$-dimensional case and for a general truncation point $\bm{t}$. As it turns
out, right-truncated Archimedean copulas are so-called \emph{tilted Archimedean copulas}, that is Archimedean
copulas with a \emph{tilted generator} $\tpsi$ of the form
\begin{align*}
  \tpsi(t)=\frac{\psi(t+h)}{\psi(h)},\quad t\in[0,\infty),
\end{align*}
for some \emph{tilt} $h\ge 0$. Also note that in what follows we interpret a sum of
the form $\sum_{i=1}^n a_i + b$ as $(\sum_{i=1}^n a_i) + b$.
\begin{theorem}[Right-truncated Archimedean copulas]\label{thm:rtrunc:AC}
  Let $C$ be a $d$-dimensional Archimedean copula with generator $\psi\in\Psi_d$. For $\bm{t}\in(0,1]^d$ such that $C(\bm{t})>0$,
  let
  \begin{align*}
    \tpsi(t)=\frac{\psi(t+\psii[C(\bm{t})])}{C(\bm{t})},\quad t\in[0,\infty),
  \end{align*}
  with corresponding $\tpsii[u]=\psii[C(\bm{t})u]-\psii[C(\bm{t})]$, $u\in[0,1]$.
  Then the right-truncated copula at $\bm{t}$ is given by
  \begin{align*}
    C_{\bm{t}}(\bm{u})=\frac{\psi\bigl(\sum_{j=1}^d \psii[C(\bm{t})u_j]-(d-1)\psii[C(\bm{t})]\bigr)}{C(\bm{t})}=\tpsi\biggl(\,\sum_{j=1}^d\tpsii[u_j]\biggr),\quad\bm{u}\in[0,1]^d,
  \end{align*}
  that is $C_{\bm{t}}$ is Archimedean with tilted generator $\tpsi(t)=\psi(t+h)/\psi(h)$ %
  and tilt $h=\psii[C(\bm{t})]$.
\end{theorem}
\begin{proof}
  For $j\in\{1,\dots,d\}$, $C(x_j;\bm{t}_{-j})=\psi\bigl(\psii[x_j]+\sum_{\tilde{j}\neq j}\psii[t_{\tilde{j}}]\bigr)$ and thus
  $C^{-1}[y_j;\bm{t}_{-j}]=\psi\bigl(\psii[y_j]-\sum_{\tilde{j}\neq j}\psii[t_{\tilde{j}}]\bigr)$.
  By Proposition~\ref{prop:rtrunc:C}, we thus have that for all $\bm{u}\in[0,1]^d$,
  {\allowdisplaybreaks%
  \begin{align}
    C_{\bm{t}}(\bm{u})&=\frac{C(\{C^{-1}[C(\bm{t})u_j;\bm{t}_{-j}]\}_j)}{C(\bm{t})}=\frac{\psi\bigl(\sum_j\psii\bigl[\psi(\psii[C(\bm{t})u_j]-\sum_{\tilde{j}\neq j}\psii[t_{\tilde{j}}])\bigr]\bigr)}{C(\bm{t})}\notag\\
                      &=\frac{\psi\bigl(\sum_{j=1}^d(\psii[C(\bm{t})u_j]-\sum_{\tilde{j}\neq j}\psii[t_{\tilde{j}}])\bigr)}{C(\bm{t})}\notag\\
                      &=\frac{\psi\bigl(\sum_{j=1}^d \psii[C(\bm{t})u_j]-\sum_{j=1}^d\sum_{\tilde{j}\neq j}\psii[t_{\tilde{j}}])}{C(\bm{t})}\notag\\
                      &=\frac{\psi\bigl(\sum_{j=1}^d \psii[C(\bm{t})u_j]-(d-1)\sum_{j=1}^d\psii[t_j])}{C(\bm{t})}\label{eq:AC:exch}\\
                      &=\frac{\psi\bigl(\sum_{j=1}^d \psii[C(\bm{t})u_j]-(d-1)\psii[C(\bm{t})]\bigr)}{C(\bm{t})}\notag\\
                      &=\frac{\psi\bigl(\sum_{j=1}^d (\psii[C(\bm{t})u_j]-\psii[C(\bm{t})])+\psii[C(\bm{t})]\bigr)}{C(\bm{t})}=\tpsi\biggl(\,\sum_{j=1}^d\tpsii[u_j]\biggr).\qedhere
  \end{align}}%
\end{proof}
\begin{remark}\label{rem:main:AC}
  \begin{enumerate}
  \item Tilted Archimedean copulas were introduced and studied in
    \cite[Section~4.2.1]{hofert2010c} under the assumption of $\psi\in\Psi_\infty$.
    By Bernstein's Theorem, see \cite[pp.~439]{feller1971}, it is well known that $\psi\in\Psi_\infty$ if and only if
    $\psi$ is the Laplace--Stieltjes transform of a distribution function
    $F$ on the positive real line known as \emph{frailty distribution} (in
    short: $\psi=\LS[F]$ or $F=\LSi[\psi]$); such
    $\psi$ are \emph{completely monotone} (that is $(-1)^k\psi^{(k)}(t)\ge
    0$, $t\in(0,\infty)$, for all
    $k\in\IN_0$) and many Archimedean generators fulfill this property (see below for examples).
    If $\psi\in\Psi_\infty$ then
    \begin{align*}
      \tpsi(t)=\frac{\psi(t+h)}{\psi(h)}=\int_0^\infty\exp(-tv)\frac{\exp(-hv)}{\psi(h)}\,\rd F(v)=\int_0^\infty\exp(-tv)\,\rd \tilde{F}(v),
    \end{align*}
    so that $\tilde{F}=\LSi[\tpsi]$. If $F$ has density $f$ (probability mass function $(p_k)_{k\in\IN}$),
    then $\tilde{F}$ has the \emph{exponentially tilted} density $\tilde{f}$ (probability mass function $(\tilde{p}_k)_{k\in\IN}$)
    given by
    \begin{align}
      \tilde{f}(v)=\frac{\exp(-hv)}{\psi(h)}f(v),\ v\in(0,\infty),\quad\biggl(\tilde{p}_k=\frac{\exp(-hk)}{\psi(h)}p_k,\ k\in\IN\biggr).\label{eq:exp:tilted:dens:pmf}
    \end{align}
    This explains the ``tilted'' in the name of tilted Archimedean copulas.
    Note that for right-truncated Archimedean copulas,
    $\psi(h)=C(\bm{t})$, in particular, neither the order of the truncation
    points nor whether they are all equal (as long as
    $C(\bm{t})$ remains constant) affects $h$ and thus how much $\psi$ is tilted.
  \item Knowing a sampling algorithm for the frailty distribution $F$ is crucial for
    efficiently sampling Archimedean copulas with the so-called Marshall--Olkin
    algorithm; see \cite{marshallolkin1988}.  This is the case for many
    well-known Archimedean families; see, for example,
    \cite[Section~2.4]{hofert2010c}. For sampling the corresponding tilted
    Archimedean copulas, one can in principle use rejection based on $F$,
    however, the resulting rejection constant is $1/\psi(h)$, which can be large.
    Under the assumption that $\psi^{1/m}\in\Psi_\infty$ for all $m\in\IN$ and that there is a sampling algorithm
    for $\psi^{1/m}$, a \emph{fast rejection algorithm} with rejection constant $\log(1/\psi(h))$ is derived
    in %
    \cite[Section~4.2.1]{hofert2010c}; %
    see also \cite{hofert2011a} and \cite{hofert2012a}. The algorithm is
    available in the \R\ package \texttt{copula} for the implemented Archimedean
    families.  This allows one to more efficiently sample from right-truncated
    Archimedean copulas in comparison to Algorithm~\ref{alg:crude:reject} which
    becomes slow especially in large dimensions if at least one component of
    $\bm{t}$ is small.
  \item As a consequence of Theorem~\ref{thm:rtrunc:AC},
    right-truncated Archimedean copulas -- as Archimedean copulas -- are
    exchangeable, independently of the choice of truncation point
    $\bm{t}\in(0,1]^d$. This is rather unexpected in the light of
    Proposition~\ref{prop:rtrunc:C} because
    $C(\cdot\,;\bm{t}_{-j})$,
    $j=1,\dots,d$, are not all equal (and so aren't
    $C^{-1}[\cdot\,;\bm{t}_{-j}]$,
    $j=1,\dots,d$) unless all truncation points are equal. The reason why
    right-truncated Archimedean copulas are exchangeable follows from
    Step~\eqref{eq:AC:exch} in the proof of Theorem~\ref{thm:rtrunc:AC} since
    the summation over all $j$ makes the sums $\sum_{\tilde{j}\neq
      j}\psii[t_{\tilde{j}}]$ (which can differ for different $j$'s) equal.
  \item\label{lim:AC:t:to:0} Another advantage of knowning that right-truncated Archimedean copulas
    are tilted Archimedean copulas is that we can easily obtain the limit for
    $\bm{t}\downarrow\bm{0}$ (that is, $t_j\downarrow
    0$ for at least one
    $j\in\{1,\dots,d\}$) under the assumption that $\psi\in\RV_{-1/\theta}^\infty$ for $\theta>0$
    (that is $\psi$ is regularly varying at
    infinity with index $-1/\theta$). %
    This result is known since \cite{larssonneslehova2011}
    (\cite{juriwuethrich2002} and \cite{charpentiersegers2007} considered the
    bivariate case and assumed equal truncation points). To see it through the lens
    of tilted Archimedean copulas, let $\psi\in\RV_{-1/\theta}^\infty$ for $\theta>0$,
    let $\tpsi$ be as in Theorem~\ref{thm:rtrunc:AC} and, without loss of generality, let $h=\psii[C(\bm{t})]>0$.
    Most importantly, we can now use that $\tpsi$ and $\tpsi(ht)$ generate the same copula and that
    \begin{align*}
      \lim_{\bm{t}\downarrow\bm{0}}\tpsi(ht)=\lim_{h\to\infty}\frac{\psi(ht+h)}{\psi(h)}=\lim_{h\to\infty}\frac{\psi(h(1+t))}{\psi(h)}=(1+t)^{-1/\theta},
    \end{align*}
    which is a generator of a Clayton copula with parameter
    $\theta$. Therefore,
    $\lim_{\bm{t}\downarrow\bm{0}}C_{\bm{t}}$ equals a Clayton copula with
    parameter
    $\theta$ in this case. Note that this result is only of limited use as
    several well-known Archimedean generators are not regularly varying with
    negative index (see later) and thus a limiting Clayton copula is not an adequate
    model in these situations.
  \item For a formula for Kendall's tau for tilted Archimedean copulas, see
    \cite[Section~4.2.1]{hofert2010c}.
  \end{enumerate}
\end{remark}

The following corollary shows that if an Archimedean copula $C$ admits
a density (which is the case for many well-known examples), then so does its
right-truncated version $C_{\bm{t}}$. Moreover, it also reveals that the
density of $C_{\bm{t}}$ is numerically as tractable as that of $C$, an important
property for (log-)likelihood-based inference of right-truncated Archimedean (and thus Archimedean)
copulas.
\begin{corollary}[Densities of right-truncated Archimedean copulas]
  Let $C$ be a $d$-dimensional Archimedean copula
  with generator $\psi\in\Psi_d$ that is $d$ times continuously differentiable.
  Then the corresponding right-truncated copula $C_{\bm{t}}$ admits the density
  \begin{align*}
    c_{\bm{t}}(\bm{u})=\frac{\tpsi^{(d)}(\sum_{j=1}^d\tpsii[u_j])}{\prod_{j=1}^d\tpsi'(\tpsii[u_j])},\quad\bm{u}\in(0,1)^d,
  \end{align*}
  where $\tpsi^{(d)}(t)=\psi^{(d)}(t+h)/\psi(h)$.
  Numerically stable proper logarithms of $\psi^{(d)}$ are available in many cases;
  see \cite{hofertmaechlermcneil2012} or the \R\ package \texttt{copula} for more
  details.
\end{corollary}

Archimedean copulas are especially of interest due to their ability to capture
tail dependence. The following lemma provides formulas for the coefficients of
tail dependence for right-truncated Archimedean copulas.
\begin{lemma}[Tail dependence for right-truncated Archimedean copulas]\label{lem:td:rtrunc:AC}
  Let $C$ be a bivariate Archimedean copula with generator $\psi\in\Psi_2$ and
  let $\bm{t}=(t_1,t_2)\in(0,1]^2$. Assuming the limits to exist, the
  coefficients of lower and upper tail dependence of the right-truncated
  Archimedean copula $C_{\bm{t}}$ %
  are given by
  \begin{align*}
    \lambda_{\text{l}}^{C_{\bm{t}}}=\lim_{t\uparrow\infty}\frac{\psi(2t+h)}{\psi(t+h)}\quad\text{and}\quad \lambda_{\text{u}}^{C_{\bm{t}}}=2-\lim_{t\downarrow 0}\frac{1-\psi(2t+h)/\psi(h)}{1-\psi(t+h)/\psi(h)}.
  \end{align*}
  If $\psi$ is differentiable and the limits exist, then
  \begin{align*}
    \lambda_{\text{l}}^{C_{\bm{t}}}=2\lim_{t\uparrow\infty}\frac{\psi'(2t+h)}{\psi'(t+h)}\quad\text{and}\quad \lambda_{\text{u}}^{C_{\bm{t}}}=2-2\lim_{t\downarrow 0}\frac{\psi'(2t+h)}{\psi'(t+h)}.
  \end{align*}
\end{lemma}
\begin{proof}
  For the coefficient of lower tail dependence, we have
  \begin{align*}
    \lambda_{\text{l}}^{C_{\bm{t}}}=\lim_{u\downarrow 0}\frac{C_{\bm{t}}(u,u)}{u}=\lim_{u\downarrow 0}\frac{\tpsi(2\tpsii[u])}{u}=\lim_{t\uparrow\infty}\frac{\tpsi(2t)}{\tpsi(t)}=\lim_{t\uparrow\infty}\frac{\psi(2t+h)}{\psi(t+h)}
  \end{align*}
  and, if $\psi$ is differentiable, an application of l'H\^opital's rule leads to $\lambda_{\text{l}}^{C_{\bm{t}}}=2\lim_{t\uparrow\infty}\frac{\psi'(2t+h)}{\psi'(t+h)}$.

  For the coefficient of upper tail dependence, we have
  \begin{align*}
    \lambda_{\text{u}}^{C_{\bm{t}}}&=\lim_{u\uparrow 1}\frac{1-2u+C_{\bm{t}}(u,u)}{1-u}=\lim_{u\uparrow1}\frac{1-2u+\tpsi(2\tpsii(u))}{1-u}=2-\lim_{u\uparrow1}\frac{1-\tpsi(2\tpsii(u))}{1-u}\\
    &=2-\lim_{t\downarrow0}\frac{1-\tpsi(2t)}{1-\tpsi(t)}=2-\lim_{t\downarrow0}\frac{1-\psi(2t+h)/\psi(h)}{1-\psi(t+h)/\psi(h)}
  \end{align*}
  and, if $\psi$ is differentiable, an application of l'H\^opital's rule leads to $\lambda_{\text{u}}^{C_{\bm{t}}}=2-2\lim_{t\downarrow 0}\frac{\psi'(2t+h)}{\psi'(t+h)}$.
\end{proof}
The formulas in Lemma~\ref{lem:td:rtrunc:AC} generalize those known for
Archimedean copulas (take $h=0$). They allow us to derive the coefficients of
tail dependence for any truncation point, regardless of whether the truncation
takes place in a symmetric manner ($t_1=t_2$) or not. This simply comes from the
fact that tilted Archimedean copulas are Archimedean copulas themselves and thus
exchangeable, the truncation point only enters the generator as a parameter
through the tilt $h=\psii[C(\bm{t})]$; compare with Example~\ref{ex:surv:G}. %

In line with intuition, the following result shows that a majority of
right-truncated Archimedean copulas satisfy
$\lambda_{\text{l}}^{C_{\bm{t}}}=\lambda_{\text{l}}^C$ and
$\lambda_{\text{u}}^{C_{\bm{t}}}=0$.
\begin{proposition}[Preserving lower tail dependence, cutting out the upper one]\label{prop:td:rtrunc:AC}
  Let $C$ be a bivariate Archimedean copula with generator $\psi\in\Psi_2$ and
  let $\bm{t}=(t_1,t_2)\in(0,1]^2$.
  \begin{enumerate}
  \item If $C$ has coefficient of lower tail dependence $\lambda_{\text{l}}^C$
    and $\lim_{t\uparrow\infty}\frac{\psi(t)}{\psi(t+h/2)}$ exists, then
    $\lambda_{\text{l}}^{C_{\bm{t}}}\ge \lambda_{\text{l}}^C$. Furthermore, if
    $\psi\in\RV_{-\alpha}$ for $\alpha>0$, then
    $\lambda_{\text{l}}^{C_{\bm{t}}}=\lambda_{\text{l}}^C$.
  \item\label{prop:td:rtrunc:AC:2} If $\psi$ is differentiable and $\bm{t}\neq(1,1)$, then $\lambda_{\text{u}}^{C_{\bm{t}}}=0$.
  \end{enumerate}
\end{proposition}
\begin{proof}
  \begin{enumerate}
  \item By Lemma~\ref{lem:td:rtrunc:AC},
    \begin{align*}
      \lambda_{\text{l}}^{C_{\bm{t}}}=\lim_{t\uparrow\infty}\frac{\psi(2t+h)}{\psi(t+h)}=\lim_{t\uparrow\infty}\frac{\psi(2(t+h/2))}{\psi(t+h/2)}\lim_{t\uparrow\infty}\frac{\psi(t+h/2)}{\psi(t+h/2 + h/2)}=\lambda_{\text{l}}^C\lim_{t\uparrow\infty}\frac{\psi(t)}{\psi(t+h/2)}.
    \end{align*}
    In particular, since $\psi(t+h/2)\le \psi(t)$, we have $\lambda_{\text{l}}^{C_{\bm{t}}}\ge\lambda_{\text{l}}^C$. For the remaining part,
    let $\eps>0$ and note that for all $t$ sufficiently large, $\psi((1+\eps)t)\le\psi(t+h/2)\le\psi(t)$.
    If $\psi\in\RV_{-\alpha}$ for $\alpha>0$, we obtain that
    \begin{align*}
      1\le\lim_{t\uparrow\infty}\frac{\psi(t)}{\psi(t+h/2)}\le \lim_{t\uparrow\infty}\frac{\psi(t)}{\psi((1+\eps)t)}=\frac{1}{\lim_{t\uparrow\infty}\frac{\psi((1+\eps)t)}{\psi(t)}}=\frac{1}{(1+\eps)^{-\alpha}}=(1+\eps)^\alpha
    \end{align*}
    which converges to $1$ for $\eps\downarrow 0$ and thus $\lambda_{\text{l}}^{C_{\bm{t}}}=\lambda_{\text{l}}^C$.
  \item If $\psi$ is differentiable and $\bm{t}\neq(1,1)$, then
    $h=\psii[C(\bm{t})]>0$ and thus $\psi'(h)\in(-\infty,0)$ which, by Lemma~\ref{lem:td:rtrunc:AC}, implies that
    $\lambda_{\text{u}}^{C_{\bm{t}}}=2-2\psi'(h)/\psi'(h)=0$.
  \end{enumerate}
\end{proof}

Many properties of right-truncated Archimedean copulas follow immediately from
the fact that they are tilted and thus Archimedean copulas. For example, the Kendall
distribution $K_{C_{\bm{t}}}(u)=\P(C_{\bm{t}}(\bm{U}_{\bm{t}})\le u)$ (for
$\bm{U}_{\bm{t}}\sim C_{\bm{t}}$) of a right-truncated Archimedean copula $C$
with tilt $h=\psii[C(\bm{t})]$ is given by %
\begin{align*}
  K_{C_{\bm{t}}}(u)&=\sum_{k=0}^{d-1}(\psii[\psi(h)u] - h)^k\frac{(-1)^k\psi^{(k)}(\psii[\psi(h)u])}{k!\psi(h)}\\ %
  &=\sum_{k=0}^{d-1}(\psii[C(\bm{t})u] - \psii[C(\bm{t})])^k\frac{(-1)^k\psi^{(k)}(\psii[C(\bm{t})u])}{k!C(\bm{t})},\quad u\in[0,1];
\end{align*}%
the formulas of \cite{barbegenestghoudiremillard1996} and \cite{mcneilneslehova2009} are obtained for $\bm{t}=\bm{1}$.
Note that the appearing generator derivatives are known for several well-known Archimedean families, see \cite{hofertmaechlermcneil2012},
and so $K_{C_{\bm{t}}}$ can be computed for such families.

\subsection{Examples}
Let us now turn to specific examples of right-truncated Archimedean copulas.
\begin{example}[Right-truncated Clayton copulas]\label{ex:clayton}
  The generator $\psi(t)=(1+t)^{-1/\theta}$, $\theta>0$, %
  of a Clayton copula has inverse $\psii[u]=u^{-\theta}-1$. We thus obtain that
  \begin{align*}
    \tpsi(t)&=\frac{\psi(t+\psii[C(\bm{t})])}{C(\bm{t})}=\frac{(1+t+\psii[C(\bm{t})])^{-1/\theta}}{C(\bm{t})}=\frac{(t+C(\bm{t})^{-\theta})^{-1/\theta}}{C(\bm{t})}\\
    &=(1+t/C(\bm{t})^{-\theta})^{-1/\theta}=\psi(t/C(\bm{t})^{-\theta}).
  \end{align*}
  Therefore, $\tpsi$ is of the form $\psi(ct)$ for some $c>0$ and thus generates
  the same Archimedean copula as $\psi$, namely a Clayton copula with parameter
  $\theta$. In other words, right-truncated Clayton copulas are again Clayton
  copulas with the same parameter. It is then no surprise that the coefficients
  of tail dependence of $C_{\bm{t}}$ are those of the Clayton copula $C$ in this
  case, which we can also easily verify from Proposition~\ref{prop:td:rtrunc:AC}
  since $\psi\in\RV_{-1/\theta}$ and $\psi$ is differentiable.
\end{example}

\begin{remark}\label{rem:limit:clayton:not:adequate}
  The fact that right-truncated Clayton copulas are again Clayton copulas (even
  with the same parameter) is one reason for the popularity of Clayton copulas in
  insurance applications, although, as we showed in Theorem~\ref{thm:rtrunc:AC},
  all right-truncated Archimedean copulas are tilted Archimedean and thus enjoy
  the corresponding tractability. Furthermore, all generators in the remaining
  examples below are not regularly varying in the sense of
  Remark~\ref{rem:main:AC}~\ref{lim:AC:t:to:0} and thus a limiting Clayton model
  is neither adequate, nor necessary to consider. For example, the following two
  examples will show that both right-truncated Ali--Mikhail--Haq and
  right-truncated Frank copulas are again Ali--Mikhail--Haq and Frank copulas,
  respectively (albeit with different parameters) and so one can work with an
  exact (instead of a limiting) model in these cases. Also right-truncated
  Gumbel and Joe copulas are tractable as we will see.
\end{remark}

\begin{example}[Right-truncated Ali--Mikhail--Haq copulas] %
  The generator $\psi(t)=(1-\theta)/(\exp(t)-\theta)$, $\theta\in[0,1)$, of an
  Ali--Mikhail--Haq copula has tilted generator
  $\tpsi(t)=\psi(t+h)/\psi(h)=\frac{1-e^{-h}\theta}{\exp(t)-e^{-h}\theta}$,
  $h\ge0$, which is of the same form as $\psi$ if $\theta$ is replaced by
  $e^{-h}\theta$. So unlike in the Clayton case, right-truncated
  Ali--Mikhail--Haq copulas are not the same Ali--Mikhail--Haq copulas, however,
  they are Ali-Mikhail-Haq copulas again, with parameter tilted by $e^{-h}$. The
  frailty distribution for Ali--Mikhail--Haq copulas is the geometric
  distribution $\Geo(1-\theta)$ on $\IN$. The frailty distribution for
  right-truncated Ali--Mikhail--Haq copulas is thus $\Geo(1-e^{-h}\theta)$ on
  $\IN$, which could have also been derived from~\eqref{eq:exp:tilted:dens:pmf};
  we demonstrate this in the next example for Frank copulas. Note
  that for $h=\psii[C(\bm{t})]$,
  $e^{-h}=e^{-\psii[C(\bm{t})]}=\frac{C(\bm{t})}{1-\theta(1-C(\bm{t}))}$, so
  that right-truncated Ali--Mikhail--Haq copulas with truncation point $\bm{t}$
  have frailty distribution
  $\Geo\bigl(\frac{1-\theta}{1-\theta(1-C(\bm{t}))}\bigr)$ on $\IN$ and thus can
  easily be sampled.
\end{example}

\begin{example}[Right-truncated Frank copulas]
  For $p=1-e^{-\theta}$, the generator $\psi(t)=-\log(1-p\exp(-t))/\theta$, %
  $\theta\in(0,\infty)$, of a Frank copula has corresponding logarithmic frailty
  distribution $\Log(p)$ with probability mass function $p_k=p^k/(-\log(1-p)k)$
  at $k\in\IN$; for a sampling algorithm, see the algorithm ``LK'' of \cite{kemp1981}. Using that
  $e^{-h}=e^{-\psii[C(\bm{t})]}=(1-e^{-\theta C(\bm{t})})/p$, it follows
  from~\eqref{eq:exp:tilted:dens:pmf} that the frailty distribution of a
  right-truncated Frank copula has probability mass function
  \begin{align*}
    \tilde{p}_k=\frac{(pe^{-h})^k}{-k\log(1-p)C(\bm{t})}=\frac{(1-e^{-\theta C(\bm{t})})^k}{k\theta C(\bm{t})}=\frac{\tilde{p}^k}{-k\log(1-\tilde{p})}\quad\text{for}\ \tilde{p}=1-e^{-\theta C(\bm{t})}.
  \end{align*}
  This is the probability mass function of a $\Log(\tilde{p})$
  distribution. We see that right-truncated Frank copulas are again Frank copulas, with
  parameter $\theta$ replaced by $\theta C(\bm{t})$, and thus can easily be sampled.
\end{example}

\begin{example}[Right-truncated Gumbel copulas]
  The generator $\psi(t)=\exp(-t^{1/\theta})$, $\theta\ge1$, %
  of a Gumbel copula
  has inverse $\psii[u]=(-\log u)^\theta$. With $h=\psii[C(\bm{t})]$, we obtain that
  \begin{align*}
    \tpsi(t)&=\frac{\psi(t+\psii[C(\bm{t})])}{C(\bm{t})}=\frac{\psi(t+h)}{\psi(h)}=\frac{\exp(-(t+h)^{1/\theta})}{\exp(-h^{1/\theta})} = \exp(-((t+h)^{1/\theta}-h^{1/\theta})),
  \end{align*}
  which is the Laplace--Stieltjes transform of the exponentially tilted stable distribution
  $\tS(1/\theta,1,\cos^\theta(\pi/(2\theta)), \I_{\{\theta=1\}},h\I_{\{\theta\neq 1\}};1)$; see
  \cite[Section~4.2.2]{hofert2010c}, \cite{hofert2011a}, and the \R\ package \texttt{copula} for more details
  about this distribution. In particular, this distribution and thus the corresponding
  right-truncated Gumbel copulas can be easily be sampled.

  Gumbel copulas $C$ are the only Archimedean extreme value copulas, so this
  example also shows that right-truncated Gumbel copulas $C_{\bm{t}}$ cannot be
  extreme value (unless $\bm{t}=\bm{1}$); this is in line with
  Proposition~\ref{prop:scaling:EVCs}. Also, by Lemma~\ref{lem:td:rtrunc:AC} %
  and Proposition~\ref{prop:td:rtrunc:AC}~\ref{prop:td:rtrunc:AC:2},
  $C_{\bm{t}}$ has no lower or upper tail dependence.
\end{example}

\begin{example}[Right-truncated Joe copulas]
  The generator $\psi(t)=1-(1-\exp(-t))^{1/\theta}$, $\theta\in[1,\infty)$, %
  of a Joe copula
  has a Sibuya ($\Sib(1/\theta)$) frailty distribution with probability mass function $p_k=\binom{1/\theta}{k}(-1)^{k-1}$ at $k\in\IN$.
  An efficient sampling algorithm for $\Sib(1/\theta)$ distributions was introduced in \cite[Proposition~3.2]{hofert2011a}; use $\theta_0=1$ there.
  By~\eqref{eq:exp:tilted:dens:pmf}, the probability mass function of the frailty distribution of the corresponding right-truncated
  Joe copula is an exponentially tilted $\Sib(1/\theta)$ distribution
  with probability mass function %
  $\tilde{p}_k=\binom{1/\theta}{k}(-1)^{k-1}(1-(1-C(\bm{t}))^\theta)^k/C(\bm{t})=p_k\frac{(1-(1-C(\bm{t}))^\theta)^k}{C(\bm{t})}$ at $k\in\IN$, %
  which is not a Sibuya distribution anymore, so right-truncated Joe copulas are Archimedean but not Joe copulas anymore.
  However, an efficient sampling algorithm for the frailty distribution corresponding to $(\tilde{p}_k)_{k\in\IN}$ can be given; see the following algorithm
  (choose $\alpha=1/\theta$ and $p=1-(1-C(\bm{t}))^\theta$ there). %
\end{example}

\begin{algorithm}[Sampling exponentially tilted Sibuya distributions]
  Let $\alpha\in(0,1]$ and $p_k^{\Sib(\alpha)}=\binom{\alpha}{k}(-1)^{k-1}$,
  $k\in\IN$, be the probability mass function of a $\Sib(\alpha)$ distribution with
  Laplace--Stieltjes transform
  $\psi(t)=1-(1-\exp(-t))^{\alpha}$. %
  Let $h>0$, $p=e^{-h}\in(0,1)$ and let
  $p_k^{\Log(p)}=p^k/(-\log(1-p)k)$, $k\in\IN$, be the
  probability mass function of $\Log(p)$.
  Furthermore, let $\tilde{F}$ be the distribution function with exponentially tilted $\Sib(\alpha)$ probability mass function
  \begin{align*}
    \tilde{p}_k=\frac{(e^{-h})^k}{\psi(h)}p_k^{\Sib(\alpha)}=\frac{p^k}{1-(1-p)^\alpha}p_k^{\Sib(\alpha)},\quad k\in\IN,
  \end{align*}
  and %
  Laplace--Stieltjes transform
  $\tpsi(t)=\psi(t+h)/\psi(h)=\psi(t-\log(p))/\psi(-\log(p))$. Then $\tilde{V}\sim\tilde{F}$ can be sampled as follows.
  \begin{enumerate}
  \item If $p\le -\log((1-p)^\alpha)$, independently sample $V\sim\Sib(\alpha)$, $U\sim\U(0,1)$ until $U\le p^{V-1}$ and then return $\tilde{V}=V$.
  \item And if $p> -\log((1-p)^\alpha)$, independently sample $V\sim\Log(p)$,
    $U\sim\U(0,1)$ until $Vp_V^{\Sib(\alpha)}/\alpha$ (see the proof for
    equivalent conditions) and then return $\tilde{V}=V$.
  \end{enumerate}
  The rejection constant of this algorithm overall is bounded above by $1/(1-1/e)\approx 1.5820$.
\end{algorithm}
\begin{proof}
  On the one hand,
  \begin{align*}
    \tilde{p}_k=\frac{p^k}{1-(1-p)^\alpha}p_k^{\Sib(\alpha)}\le \frac{p}{1-(1-p)^\alpha}p_k^{\Sib(\alpha)},\quad k\in\IN, %
  \end{align*}
  which implies that we can sample $\tilde{F}$ by rejection from a $\Sib(\alpha)$ distribution with rejection constant $c^{\Sib(\alpha)}=\frac{p}{1-(1-p)^\alpha}\ge 1$
  (decreasing as a function of $p$, %
  $1/\alpha$ for $p=0$, %
  $1$ for $p=1$) and acceptance condition $Uc^{\Sib(\alpha)}p_V^{\Sib(\alpha)}\le \tilde{p}_V$, equivalent to $U\le p^{V-1}$.
  On the other hand, since
  \begin{align*}
    kp_k^{\Sib(\alpha)}=k\binom{\alpha}{k}(-1)^{k-1}=k\alpha\frac{\prod_{j=1}^{k-1}(j-\alpha)}{k!}=\alpha\prod_{j=1}^{k-1}(1-\alpha/j)\le\alpha,\quad k\in\IN,
  \end{align*}
  we also have that
  \begin{align*}
    \tilde{p}_k&=\frac{p^k}{1-(1-p)^\alpha}p_k^{\Sib(\alpha)}=\frac{p^k}{(1-(1-p)^\alpha)k}kp_k^{\Sib(\alpha)}=\frac{-\log(1-p)}{1-(1-p)^\alpha} p_k^{\Log(p)} k p_k^{\Sib(\alpha)}\\
    &\le\frac{-\log((1-p)^\alpha)}{1-(1-p)^\alpha}p_k^{\Log(p)},\quad k\in\IN, %
  \end{align*}
  and so we can also sample $\tilde{F}$ by rejection from a $\Log(p)$ distribution with rejection constant $c^{\Log(p)}=\frac{-\log((1-p)^\alpha)}{1-(1-p)^\alpha}\ge 1$
  (maximal value $1/(1-1/e)$ is attained for $p=1-e^{-1/\alpha}$) and acceptance condition $Uc^{\Log(p)}p_V^{\Log(p)}\le \tilde{p}_V$, %
  equivalent to $U\le Vp_V^{\Sib(\alpha)}/\alpha=\binom{\alpha-1}{V-1}(-1)^{V-1}=\binom{V-1-\alpha}{V-1}=\prod_{j=1}^{V-1}(1-\alpha/j)$.
  Finally, note that $c^{\Sib(\alpha)}\le c^{\Log(p)}$ if and only if $p\le -\log((1-p)^\alpha)$ in which case
  also $c^{\Sib(\alpha)}$ is bounded above by $1/(1-1/e)$.
\end{proof}

\section{Right-truncated copulas related to Archimedean copulas}
In this section we cover the right-truncated copulas of outer power Archimedean
copulas, Archimax copulas with logistic stable tail dependence function and
nested Archimedean copulas.

\subsection{Right-truncated outer power Archimedean copulas}\label{sec:rtrunc:opAC}
If $\psi\in\Psi_{\infty}$, then $\opsi(t)=\psi(t^\alpha)\in\Psi_{\infty}$ for
all $\alpha\in(0,1]$.  The Archimedean copulas generated by $\opsi$ are known
as \emph{outer power Archimedean copulas}; see
\cite[Section~4.5.1]{nelsen2006} and \cite[Section~2.4]{hofert2010c}. Outer
power generators $\opsi$ have become popular in applications due to the added
flexibility through the additional parameter $\alpha$, which can be beneficial
for modeling purposes; see \cite{hofertscherer2011} or
\cite{goreckihofertokhrin2020}. If the \emph{base generator} $\psi$ is from
Clayton's family, the resulting outer power Clayton copula can reach any
lower and upper tail dependence of interest.

By Theorem~\ref{thm:rtrunc:AC}, the right-truncated copula $C_{\bm{t}}$ of
an outer power Archimedean copula with generator $\opsi(t)=\psi(t^\alpha)$, $\alpha\in(0,1]$,
is (tilted outer power) Archimedean with generator
\begin{align*}
  \tpsi(t)=\frac{\opsi(t+\opsii[\mathring{C}(\bm{t})])}{\mathring{C}(\bm{t})},
\end{align*}
where $\mathring{C}$ is the outer power Archimedean copula generated by $\opsi$.
If $F=\LSi[\psi]$ denotes the frailty distribution corresponding to $\psi$
and $V\sim F$, then $\mathring{F}=\LSi[\opsi]$ has stochastic representation
\begin{align*}
  \mathring{V}=SV^{1/\alpha}, %
\end{align*}
where $S\sim\S(\alpha,1,\cos^{1/\alpha}(\alpha\pi/2),\I_{\{\alpha=1\}};1)$ denotes the stable
distribution with Laplace--Stieltjes transform $\psi_{S}(t)=\exp(-t^\alpha)$; %
see \cite[Theorem~4.2.6]{hofert2010c}. %

Under the assumptions stated in Remark~\ref{rem:main:AC}~\ref{lim:AC:t:to:0}, it
is also straightforward to verify that $\lim_{\bm{t}\downarrow\bm{0}}C_{\bm{t}}$
for a right-truncated outer power Archimedean copula with parameter
$\alpha\in(0,1]$ is a Clayton copula with parameter $\theta/\alpha$.

\subsection{Right-truncated Archimax copulas with logistic stable tail dependence function}\label{sec:rtrunc:AXC:log:stdf}
\emph{Archimax copulas}, see \cite{caperaafougeresgenest2000} and
\cite{charpentierfougeresgenestneslehova2014}, are copulas of the form
\begin{align*}
  C(\bm{u})=\psi(\ell(\psii[u_1],\dots,\psii[u_d])),\quad\bm{u}\in[0,1]^d,
\end{align*}
where $\psi\in\Psi_d$ and $\ell:[0,\infty)^d\to[0,\infty)$ is a stable tail
dependence function; see \cite{ressel2013} %
and \cite{charpentierfougeresgenestneslehova2014} for a characterization of
stable tail dependence
functions. %
A popular family of stable tail dependence functions are \emph{logistic stable
  tail dependence functions}
$\ell(\bm{x})=\ell_\alpha(\bm{x})=\bigl(\sum_{j=1}^d
x_j^{1/\alpha}\bigr)^{\alpha}$, $\bm{x}\in[0,\infty)^d$, $\alpha\in(0,1]$; logistic stable
tail dependence functions are the stable tail dependence functions of \emph{logistic copulas}
(extreme value copulas) also known as Gumbel copulas we already encountered.

The class of outer power Archimedean copulas is equivalent to the class of
Archimax copulas with logistic stable tail dependence functions. By the above
considerations we thus also know that right-truncated Archimax copulas with
logistic stable tail dependence functions are tilted outer power Archimedean and
thus Archimedean copulas again.

For Archimax copulas with other stable tail dependence functions, the
corresponding right-truncated copulas may not be available analytically anymore
as stable tail dependence functions are typically not componentwise analytically
invertible unless in the (hierarchical) logistic case; see
Example~\ref{ex:NAC:special:cases}~\ref{ex:NAC:special:cases:hier:logistic:stdf} below
for the hierarchical case.

\subsection{Right-truncated nested Archimedean copulas}
In this section we consider \emph{nested Archimedean copulas} of the form
\begin{align}
  C(\bm{u})=C_0(C_1(\bm{u}_1),\dots,C_S(\bm{u}_S))=\psi_0\biggl(\,\sum_{s=1}^S\psiis{0}\biggl[\psi_s\biggl(\,\sum_{j=1}^{d_s}\psiis{s}[u_{sj}]\biggr)\biggr]\biggr),\label{eq:NAC}
\end{align}
where $\bm{u}=(\bm{u}_1,\dots,\bm{u}_S)\in[0,1]^d$ with
$\bm{u}_s=(u_{s1},\dots,u_{sd_s})$, $s\in\{1,\dots,S\}$, and $C_0,C_1,\dots,C_S$
are Archimedean copulas of dimensions $S,d_1,\dots,d_S$ (such that $\sum_{s=1}^Sd_s=d$) with generators
$\psi_0,\psi_1,\dots,\psi_S\in\Psi_{\infty}$, respectively. We assume the
sufficient nesting condition of \cite{mcneil2008} to hold
($(\psiis{0}\circ\psi_s)'$ being completely monotone for all $s=1,\dots,S$), so
that $C$ in \eqref{eq:NAC} is guaranteed to be a proper copula.
Furthermore, for ease of notation, let
\begin{align*}
  \bm{t}=(\bm{t}_1,\dots,\bm{t}_S)=(t_{11},\dots,t_{1d_1},\dots,t_{S1},\dots,t_{Sd_S})
\end{align*}
and, similarly to \eqref{comp:map},
\begin{align*}
  \bm{x}_s\mapsto C(\bm{x}_s;\bm{t}_{-s})=C(\bm{t}_1,\dots,\bm{t}_{s-1},\bm{x}_s,\bm{t}_{s+1},\dots,\bm{t}_S).
\end{align*}
The following result provides the right-truncated copulas of nested Archimedean copulas of form~\eqref{eq:NAC}.
\begin{theorem}[Right-truncated nested Archimedean copulas]\label{thm:right:rtrunc:NAC}
  Let $C$ be a $d$-dimensional nested Archimedean copula of form~\eqref{eq:NAC}.
  For $\bm{t}\in(0,1]^d$ such that $C(\bm{t})>0$, the right-truncated copula at $\bm{t}$, that is $C_{\bm{t}}(\bm{u})$, is given by
  \begin{align}
    \frac{\psi_0\bigl(\,\sum_{s=1}^S\psiis{0}\bigl[\psi_s\bigl(\,\sum_{j=1}^{d_s}
      \psiis{s}\bigl[\psi_0\bigl(\psiis{0}[C(\bm{t})u_{sj}] - \psiis{0}[C(\bm{1};\bm{t}_{-s})]\bigr)\bigr] - (d_s-1)\psiis{s}[C_s(\bm{t}_s)]
      \bigr)\bigr]\bigr)}{C(\bm{t})},\label{eq:right:trunc:NAC}
  \end{align}
  where $\psiis{0}[C(\bm{1};\bm{t}_{-s})]=\psiis{0}[C(\bm{t})]-\psiis{0}[C_{s}(\bm{t}_{s})]$.
\end{theorem}
\begin{proof}
  We have that
  \begin{align*}
    C(x_{sj};\bm{t}_{-sj})=\psi_0\biggl(\,\sum_{\tb{\tilde{s}=1}{\tilde{s}\neq s}}^S\psiis{0}[C_{\tilde{s}}(\bm{t}_{\tilde{s}})]+\psiis{0}\biggl[\psi_s\biggl(\,\sum_{\tb{\tilde{j}=1}{\tilde{j}\neq j}}^{d_s}\psiis{s}[t_{s\tilde{j}}] + \psiis{s}[x_{sj}]\biggr)\biggr]\biggr)
  \end{align*}
  and thus
  \begin{align*}
    C^{-1}[y_{sj};\bm{t}_{-sj}]&=\psi_s\biggl(\psiis{s}\biggl[\psi_0\biggl(\psiis{0}[y_{sj}] - \sum_{\tb{\tilde{s}=1}{\tilde{s}\neq s}}^S\psiis{0}[C_{\tilde{s}}(\bm{t}_{\tilde{s}})]\biggr)\biggr] - \sum_{\tb{\tilde{j}=1}{\tilde{j}\neq j}}^{d_s}\psiis{s}[t_{s\tilde{j}}]\biggr)\\
    &=\psi_s\biggl(\psiis{s}\biggl[\psi_0\biggl(\psiis{0}[y_{sj}] - \psiis{0}[C(\bm{1};\bm{t}_{-s})]\biggr)\biggr] - \sum_{\tb{\tilde{j}=1}{\tilde{j}\neq j}}^{d_s}\psiis{s}[t_{s\tilde{j}}]\biggr),
  \end{align*}
  where we used that $\sum_{\tb{\tilde{s}=1}{\tilde{s}\neq s}}^S\psiis{0}[C_{\tilde{s}}(\bm{t}_{\tilde{s}})]=\psiis{0}[C(\bm{1};\bm{t}_{-s})]$.
  Using $y_{sj}=C(\bm{t})u_{sj}$, $s=1,\dots,S$, $j=1,\dots,d_s$, we obtain from Proposition~\ref{prop:rtrunc:C} that $C_{\bm{t}}(\bm{u})$ equals
  \begin{align*}
    &\phantom{{}={}}\frac{C(\{C^{-1}[C(\bm{t})u_{sj};\bm{t}_{-sj}]\}_{s,j})}{C(\bm{t})}\\
                      &=\frac{\psi_0\bigl(\,\sum_{s=1}^S\psiis{0}\bigl[\psi_s\bigl(\,\sum_{j=1}^{d_s}\bigl(
                        \psiis{s}\bigl[\psi_0\bigl(\psiis{0}[C(\bm{t})u_{sj}] - \psiis{0}[C_{-s}(\bm{t}_{-s})]\bigr)\bigr] - \sum_{\tb{\tilde{j}=1}{\tilde{j}\neq j}}^{d_s}\psiis{s}[t_{s\tilde{j}}]
                        \bigr)\bigr)\bigr]\bigr)}{C(\bm{t})}. %
  \end{align*}
  Expanding the sum over $j$ and using the fact that
  \begin{align*}
    \sum_{j=1}^{d_s}\sum_{\tb{\tilde{j}=1}{\tilde{j}\neq j}}^{d_s}\psiis{s}[t_{s\tilde{j}}]=\sum_{j=1}^{d_s}(d_s-1)\psiis{s}[t_{sj}]=(d_s-1)\psiis{s}[C_s(\bm{t}_s)],
  \end{align*}
  we obtain the form of $C_{\bm{t}}$ as claimed. To see the representation for $\psiis{0}[C(\bm{1};\bm{t}_{-s})]$, note that
  $\psiis{0}[C(\bm{1};\bm{t}_{-s})]=\sum_{\tb{\tilde{s}=1}{\tilde{s}\neq s}}^S\psiis{0}[C_{\tilde{s}}(\bm{t}_s)]=\sum_{\tilde{s}=1}^S\psiis{0}[C_{\tilde{s}}(\bm{t}_{\tilde{s}})] - \psiis{0}[C_{s}(\bm{t}_s)]=\psiis{0}[C(\bm{t})]-\psiis{0}[C_{s}(\bm{t}_s)]$.
\end{proof}
It is readily checked that $C_{\bm{t}}$ in~\eqref{eq:right:trunc:NAC} is
grounded. To verify that it has standard uniform univariate margins, consider the $\tilde{s}\tilde{j}$th margin and
let $u_{sj}=1$ for all $s\neq\tilde{s}$ or $j\neq\tilde{j}$. This implies that
\begin{align}
  &\phantom{{}={}}\psi_0\bigl(\psiis{0}[C(\bm{t})u_{sj}] - \psiis{0}[C(\bm{1};\bm{t}_{-s})]\bigr)\notag\\
  &=\psi_0\bigl(\psiis{0}[C(\bm{t})u_{sj}] - \psiis{0}[C(\bm{t})] + \psiis{0}[C_{s}(\bm{t}_{s})]\bigr)\notag\\
  &=\begin{cases}
    \psi_0\bigl(\psiis{0}[C(\bm{t})u_{\tilde{s}\tilde{j}}] - \psiis{0}[C(\bm{t})] + \psiis{0}[C_{\tilde{s}}(\bm{t}_{\tilde{s}})]\bigr),&s=\tilde{s}\ \text{and}\ j=\tilde{j},\\
    C_{s}(\bm{t}_{s}),&s\neq \tilde{s}\ \text{or}\ j\neq \tilde{j}. %
  \end{cases}\label{eq:NAC:simplification}
\end{align}
It then follows from~\eqref{eq:right:trunc:NAC} that
{\allowdisplaybreaks
\begin{align*}
   &\phantom{{}={}}C_{\bm{t}}(\bm{1},u_{\tilde{s}\tilde{j}},\bm{1})C(\bm{t})\\
   &=\psi_0\biggl(\,\sum_{s=1}^S\psiis{0}\biggl[\psi_s\biggl(\,\sum_{j=1}^{d_s}
     \psiis{s}\bigl[\psi_0\bigl(\psiis{0}[C(\bm{t})u_{sj}] - \psiis{0}[C(\bm{t})] + \psiis{0}[C_s(\bm{t}_s)]\bigr)\bigr]\\
     &\hphantom{=\psi_0\biggl(\,\sum_{s=1}^S\psiis{0}\biggl[\psi_s\biggl(\,}
     - (d_s-1)\psiis{s}[C_s(\bm{t}_s)]\biggr)\biggr]\biggr)\\
   &=\psi_0\biggl(\,\sum_{\tb{s=1}{s\neq\tilde{s}}}^S\psiis{0}[C_s(\bm{t}_s)]
     + \psiis{0}\bigl[\psi_{\tilde{s}}\bigl((d_{\tilde{s}}-1)\psiis{{\tilde{s}}}[C_{\tilde{s}}(\bm{t}_{\tilde{s}})]\\ %
     &\hphantom{=\psi_0\biggl(\,\sum_{\tb{s=1}{s\neq\tilde{s}}}^S\psiis{0}[C_s(\bm{t}_s)]+ \psiis{0}\bigl[\psi_{\tilde{s}}\bigl(}
       + \psiis{{\tilde{s}}}\bigl[\psi_0\bigl(\psiis{0}[C(\bm{t})u_{{\tilde{s}}{\tilde{j}}}] - \psiis{0}[C(\bm{t})] + \psiis{0}[C_{\tilde{s}}(\bm{t}_{\tilde{s}})]\bigr)\bigr]\\ %
     &\hphantom{=\psi_0\biggl(\,\sum_{\tb{s=1}{s\neq\tilde{s}}}^S\psiis{0}[C_s(\bm{t}_s)]+ \psiis{0}\bigl[\psi_{\tilde{s}}\bigl(}
     - (d_{\tilde{s}}-1)\psiis{\tilde{s}}[C_{\tilde{s}}(\bm{t}_{\tilde{s}})]
     \bigr)\bigr]\biggr)\\ %
   &=\psi_0\biggl(\,\sum_{\tb{s=1}{s\neq\tilde{s}}}^S\psiis{0}[C_s(\bm{t}_s)]%
     + \psiis{0}\bigl[\psi_{\tilde{s}}\bigl(
     \psiis{{\tilde{s}}}\bigl[\psi_0\bigl(\psiis{0}[C(\bm{t})u_{{\tilde{s}}{\tilde{j}}}] - \psiis{0}[C(\bm{t})] + \psiis{0}[C_{\tilde{s}}(\bm{t}_{\tilde{s}})]\bigr)\bigr]
     \bigr)\bigr]\biggr)\\
   &=\psi_0\biggl(\,\sum_{\tb{s=1}{s\neq\tilde{s}}}^S\psiis{0}[C_s(\bm{t}_s)] %
     + \psiis{0}[C(\bm{t})u_{{\tilde{s}}{\tilde{j}}}] - \psiis{0}[C(\bm{t})] + \psiis{0}[C_{\tilde{s}}(\bm{t}_{\tilde{s}})]
     \biggr)\\
   &=\psi_0\biggl(\,\sum_{s=1}^S\psiis{0}[C_s(\bm{t}_s)]
     + \psiis{0}[C(\bm{t})u_{{\tilde{s}}{\tilde{j}}}] - \psiis{0}[C(\bm{t})]
     \biggr)\\
   &=\psi_0\bigl(\psiis{0}[C(\bm{t})]
     + \psiis{0}[C(\bm{t})u_{{\tilde{s}}{\tilde{j}}}] - \psiis{0}[C(\bm{t})]
     \bigr)=\psi_0\bigl(\psiis{0}[C(\bm{t})u_{{\tilde{s}}{\tilde{j}}}]\bigr)=C(\bm{t})u_{\tilde{s}\tilde{j}},
\end{align*}}%
from which we correctly obtain that $C_{\bm{t}}(\bm{1},u_{\tilde{s}\tilde{j}},\bm{1})=u_{\tilde{s}\tilde{j}}$.

Based on~\eqref{eq:NAC:simplification}, a similar but more tedious calculation
can be done to derive the bivariate margins of~\eqref{eq:right:trunc:NAC}.
\begin{corollary}[Bivariate margins of a right-truncated nested Archimedean copula]\label{cor:biv:margins:NAC}
  The copula $C_{\bm{t}}$ in~\eqref{eq:right:trunc:NAC} has bivariate margin $C_{\bm{t}}(\bm{1},u_{s_1j_1},\bm{1},u_{s_2j_2},\bm{1})$ given by
  \begin{align*}
    \begin{cases}
      \frac{\psi_0\bigl(\tilde{h}+\psiis{0}\bigl[\psi_s\bigl(\,\sum_{k=1}^2\psiis{s}\bigl[\psi_0\bigl(\psiis{0}[C(\bm{t})u_{sj_k}] - \tilde{h}\bigr)\bigr] - \psiis{s}[C_s(\bm{t}_s)]   \bigr)\bigr] \bigr)}{C(\bm{t})},& s_1=s_2=s,\\ %
      \tpsi_0(\tpsiis{0}[u_{s_1j_1}]+\tpsiis{0}[u_{s_2j_2}]), & s_1\neq s_2,
    \end{cases}
  \end{align*}
  where $\tilde{h}=\psiis{0}[C(\bm{t})]-\psiis{0}[C_s(\bm{t}_s)]$ %
  if $s_1=s_2=s$ and where
  $\tpsi_0(t)=\psi_0(t+h)/\psi_0(h)$ with $h=\psiis{0}[C(\bm{t})]$ if $s_1\neq s_2$.
  In particular, bivariate margins of right-truncated nested Archimedean copulas
  are tilted Archimedean copulas if the corresponding indices belong to different sectors.
  And the fact that this does not hold in general if they belong to the same sectors
  implies that a pairwise margin of a truncated nested Archimedean copula not necessarily
  equals the truncated corresponding pairwise margin of a nested Archimedean copula.
\end{corollary}

\begin{example}[Independent Archimedean copulas, exchangeable nested Archimedean
  copulas, hierarchical logistic stable tail dependence
  function]\label{ex:NAC:special:cases}
  \begin{enumerate}
  \item If $C_0$ is the independence copula (for example if $\psi_0(t)=\exp(-t)$), then
    \begin{align*}
      \psi_0\bigl(\psiis{0}[C(\bm{t})u_{sj}] - \psiis{0}[C(\bm{1};\bm{t}_{-s})]\bigr) &= \psi_0\bigl(\psiis{0}[C(\bm{t})u_{sj}] - \psiis{0}[C(\bm{t})] + \psiis{0}[C_{s}(\bm{t}_{s})]]\bigr)\\
      &=C(\bm{t})u_{sj}\frac{1}{C(\bm{t})} C_s(\bm{t}_s) = C_s(\bm{t}_s)u_{sj},
    \end{align*}
    so that
    \begin{align*}
      C_{\bm{t}}(\bm{u})&=\frac{\psi_0\bigl(\,\sum_{s=1}^S\psiis{0}\bigl[\psi_s\bigl(\,\sum_{j=1}^{d_s}
      \psiis{s}\bigl[C_s(\bm{t}_s)u_{sj}\bigr] - (d_s-1)\psiis{s}[C_s(\bm{t}_s)]
                          \bigr)\bigr]\bigr)}{C(\bm{t})},\\
      &=\frac{\prod_{s=1}^S \psi_s\bigl(\,\sum_{j=1}^{d_s}
      \psiis{s}\bigl[C_s(\bm{t}_s)u_{sj}\bigr] - (d_s-1)\psiis{s}[C_s(\bm{t}_s)]
        \bigr)}{C(\bm{t})}\\
      &=\prod_{s=1}^S \frac{\psi_s\bigl(\,\sum_{j=1}^{d_s}
      \psiis{s}\bigl[C_s(\bm{t}_s)u_{sj}\bigr] - (d_s-1)\psiis{s}[C_s(\bm{t}_s)]
        \bigr)}{C_s(\bm{t}_s)} = \prod_{s=1}^S C_{s,\bm{t}_s}(\bm{u}_s),
    \end{align*}
    where $C_{s,\bm{t}_s}$ denotes $C_s$ truncated at $\bm{t}_s$;
    compare with Example~\ref{ex:indep:dep:blocks:max}~\ref{ex:indep:dep:blocks:max:indep:blocks}. %
  \item If $\psi_0=\psi_1=\dots=\psi_S=\psi$, we obtain from~\eqref{eq:right:trunc:NAC} by canceling out compositions $\psii\circ\psi$ that $C_{\bm{t}}(\bm{u})$ equals
    \begin{align*}
      &=\frac{\psi\bigl(\,\sum_{s=1}^S\bigl(\,\sum_{j=1}^{d_s}
      \bigl(\psii[C(\bm{t})u_{sj}] - \psii[C(\bm{t})]+\psii[C_{s}(\bm{t}_{s})]\bigr) - (d_s-1)\psii[C_s(\bm{t}_s)]
        \bigr)\bigr)}{C(\bm{t})}\\
      &=\frac{\psi\bigl(\,\sum_{s=1}^S\bigl(\,\sum_{j=1}^{d_s}
        \psii[C(\bm{t})u_{sj}] - d_s\psii[C(\bm{t})]+d_s\psii[C_{s}(\bm{t}_{s})] - (d_s-1)\psii[C_s(\bm{t}_s)]
        \bigr)\bigr)}{C(\bm{t})}\\
      &=\frac{\psi\bigl(\,\sum_{s=1}^S\bigl(\,\sum_{j=1}^{d_s}
        \psii[C(\bm{t})u_{sj}] - d_s\psii[C(\bm{t})] + \psii[C_{s}(\bm{t}_{s})]
        \bigr)\bigr)}{C(\bm{t})}\\
      &=\frac{\psi\bigl(\,\sum_{s=1}^S\sum_{j=1}^{d_s}
        \psii[C(\bm{t})u_{sj}] - d\psii[C(\bm{t})] + \sum_{s=1}^S\psii[C_{s}(\bm{t}_{s})]
         \bigr)}{C(\bm{t})}\\
      &=\frac{\psi\bigl(\,\sum_{s=1}^S\sum_{j=1}^{d_s}
        \psii[C(\bm{t})u_{sj}] - d\psii[C(\bm{t})] + \psii[C(\bm{t})]
         \bigr)}{C(\bm{t})}\\
      &=\frac{\psi\bigl(\,\sum_{s=1}^S\sum_{j=1}^{d_s}
        \psii[C(\bm{t})u_{sj}] - (d-1)\psii[C(\bm{t})]
         \bigr)}{C(\bm{t})}.
    \end{align*}
    This is a right-truncated copula as in Theorem~\ref{thm:rtrunc:AC}, that is a tilted Archimedean copula,
    which is intuitive since taking equal generators in~\eqref{eq:NAC} results in an Archimedean copula. %
  \item\label{ex:NAC:special:cases:hier:logistic:stdf} If $\psi_s(t)=\opsi_s(t)=\psi(t^{\alpha_s})$, $s=0,1,\dots,S$, are outer power Archimedean generators
    with parameters $\alpha_s\in(0,1]$, $s=0,1,\dots,S$, satisfying $\alpha_0\ge\max\{\alpha_1,\dots,\alpha_S\}$ (sufficient nesting condition),
    then, by~\eqref{eq:right:trunc:NAC}, $C_{\bm{t}}(\bm{u})$ equals
    \begin{align*}
      \frac{\psi\bigl(\bigl(\,\sum_{s=1}^S\bigl(\,\sum_{j=1}^{d_s}(
      \psii[C(\bm{t})u_{sj}]^{\frac{1}{\alpha_0}}\!-\!\psii[C(\bm{1}; \bm{t}_{-s})]^{\frac{1}{\alpha_0}}
      )^{\frac{\alpha_0}{\alpha_s}}\!-\!(d_s\!-\!1)\psii[C_s(\bm{t}_s)]^{\frac{1}{\alpha_s}}\bigr)^{\frac{\alpha_s}{\alpha_0}}\bigr)^{\alpha_0}\bigr)}{C(\bm{t})}.
    \end{align*}
    This is also %
    the form of a right-truncated Archimax copula with hierarchical stable tail dependence function
    \begin{align*}
      \ell(\bm{x})=\ell_{\alpha_0}(\ell_{\alpha_1}(\bm{x}_1),\dots,\ell_{\alpha_S}(\bm{x}_S))=\biggl(\,\sum_{s=1}^S\biggl(\,\sum_{j=1}^{d_s}x_{sj}^{\frac{1}{\alpha_s}}\biggr)^{\frac{\alpha_s}{\alpha_0}}\biggr)^{\alpha_0},\quad \bm{x}=(\bm{x}_1,\dots,\bm{x}_S)\in[0,\infty)^d,
    \end{align*}
    since the choice of generators $\psi_0,\psi_1,\dots,\psi_S$ implies that
    \begin{align*}
      C(\bm{u})=\psi\biggl(\biggl(\,\sum_{s=1}^S\biggl(\,\sum_{j=1}^{d_s} \psii[u_{sj}]^{\frac{1}{\alpha_s}}\biggr)^{\frac{\alpha_s}{\alpha_0}}\biggr)^{\alpha_0}\biggr),
    \end{align*}
    which is an Archimax copula with generator $\psi$ and hierarchical stable tail dependence function $\ell$.
  \end{enumerate}
\end{example}

\begin{example}[Right-truncated nested Clayton and Gumbel copulas]
  Figure~\ref{fig:rtrunc:NACs} shows 5000 samples from truncated nested Clayton
  (left column) and truncated nested Gumbel (right column) copulas.  The nested
  copulas are of the form $C(\bm{u})=C_0(u_{11},C_1(u_{21},u_{22}))$, where
  $C_0,C_1$ are from the respective Archimedean family with parameters
  $\theta_0,\theta_1$ chosen such that the corresponding Kendall's taus are
  $0.5, 0.75$, respectively. The truncation points are $\bm{t}=(1,1,1)$ (no
  truncation; top row), $\bm{t}=(0.2, 0.5, 0.5)$ and $\bm{t}=(0.9, 0.9, 0.9)$
  (middle row), and $\bm{t}=(0.2, 0.1, 0.9)$ and $\bm{t}=(0.5, 0.5, 0.5)$ (bottom
  row).  As is clear from Corollary~\ref{cor:biv:margins:NAC} and
  Example~\ref{ex:clayton}, for the right-truncated Clayton copulas, the bivariate
  margins of pairs with indices belonging to different sectors (so $(U_1,U_2)$
  and $(U_1,U_3)$) do not change when changing the truncation point. We can also see that the
  within-sector margin (so $(U_2,U_3)$) is not much affected by a change of the
  truncation point. For the right-truncated Gumbel copulas, already moderate
  right-truncation will lead to weaker dependence as right-truncation especially
  affects the upper-right tail, in line with Proposition~\ref{prop:td:rtrunc:AC}.
  \begin{figure}[htbp]
    \centering
    \hfill
    \includegraphics[width=0.4\textwidth]{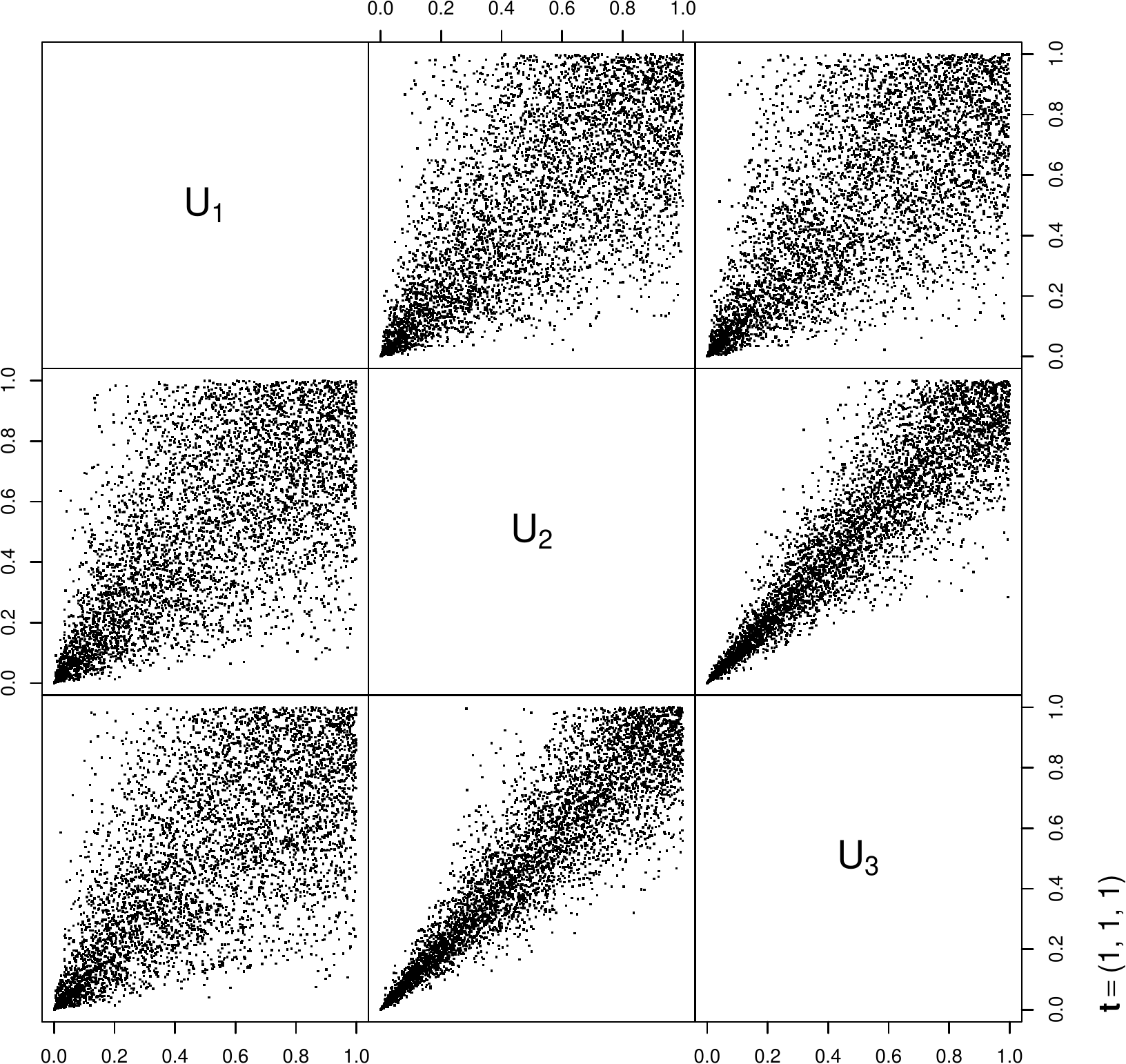}%
    \hfill
    \includegraphics[width=0.4\textwidth]{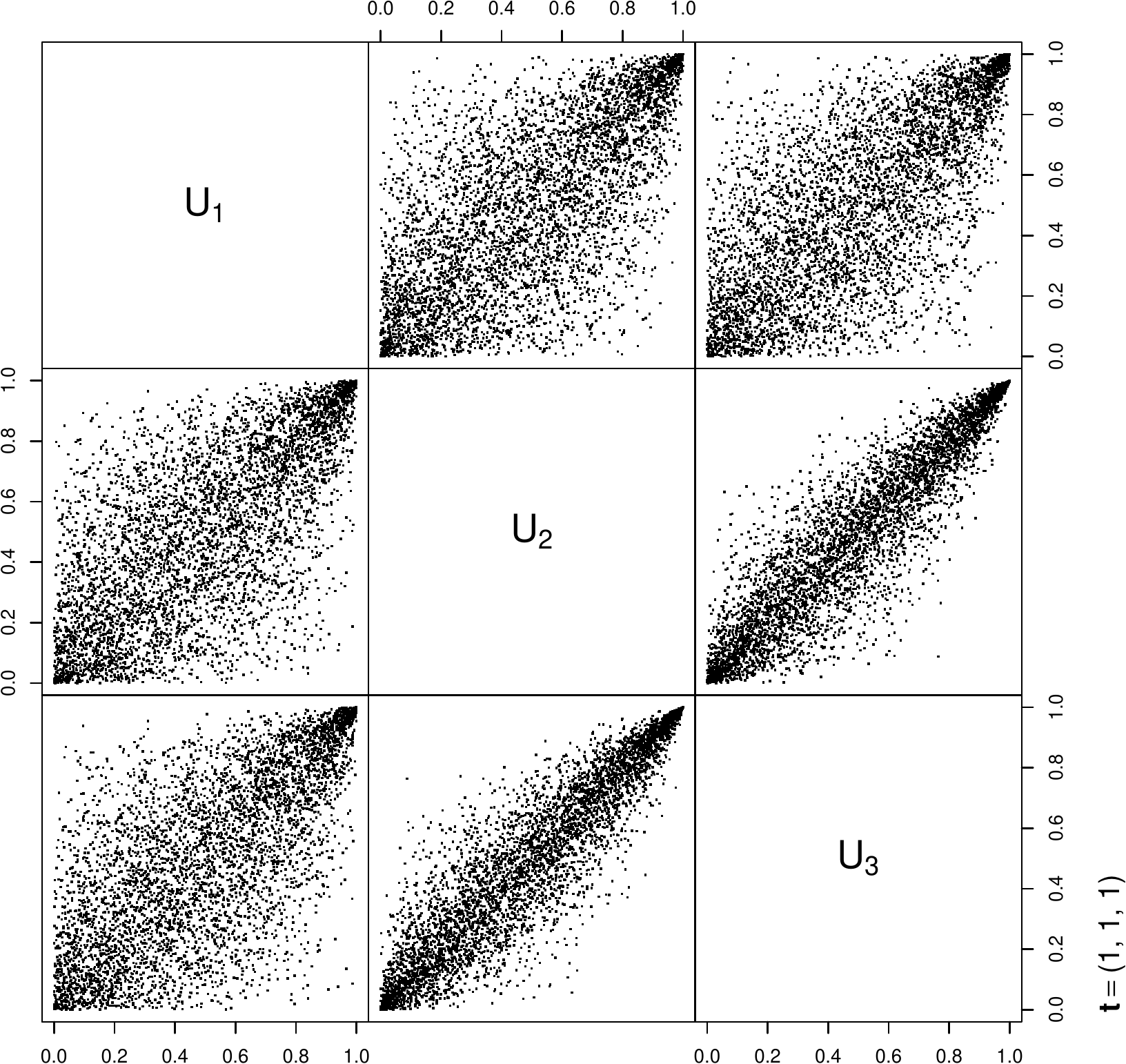}%
    \hspace*{\fill}\par\vspace{2mm}
    \hfill
    \includegraphics[width=0.4\textwidth]{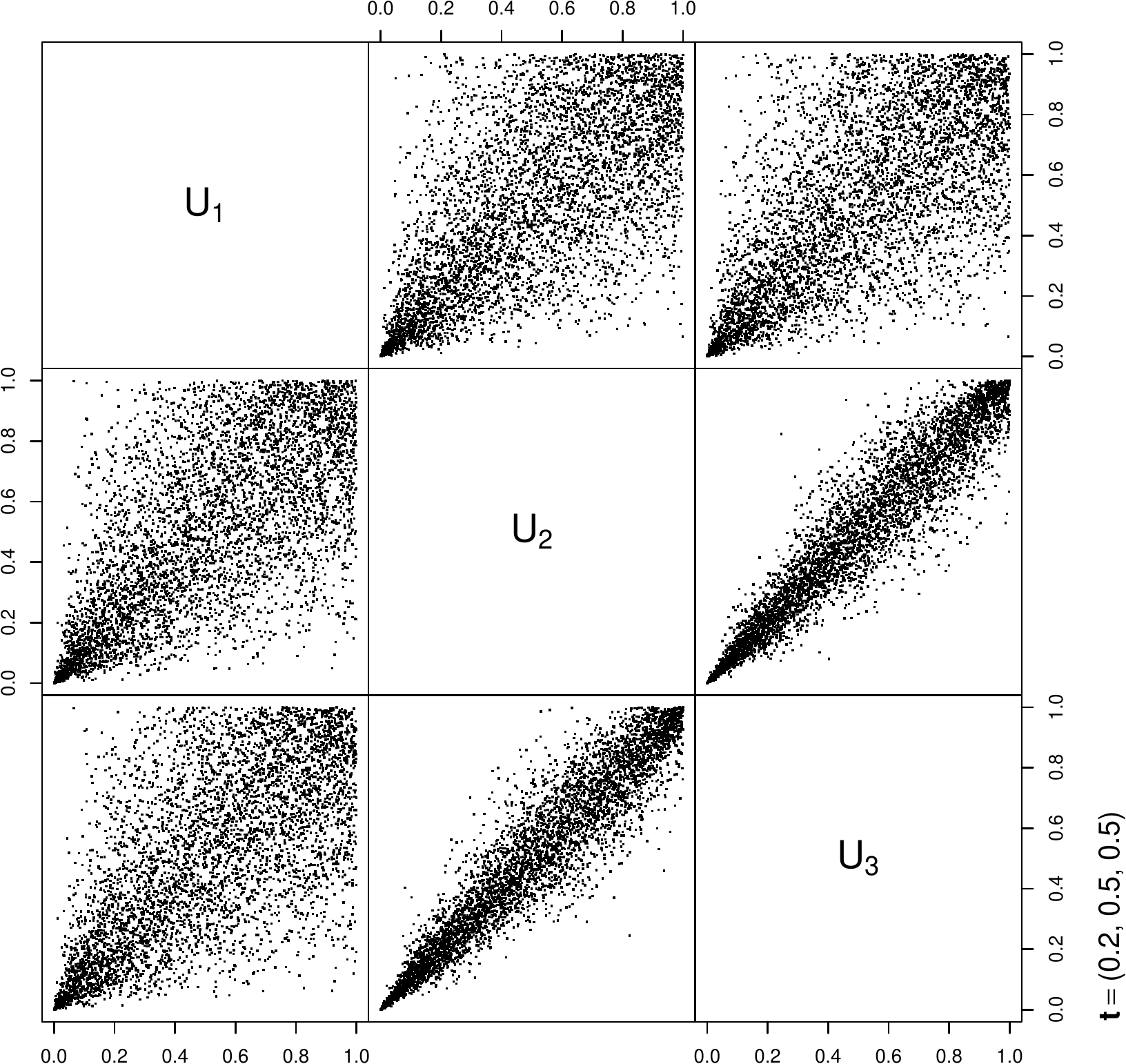}%
    \hfill
    \includegraphics[width=0.4\textwidth]{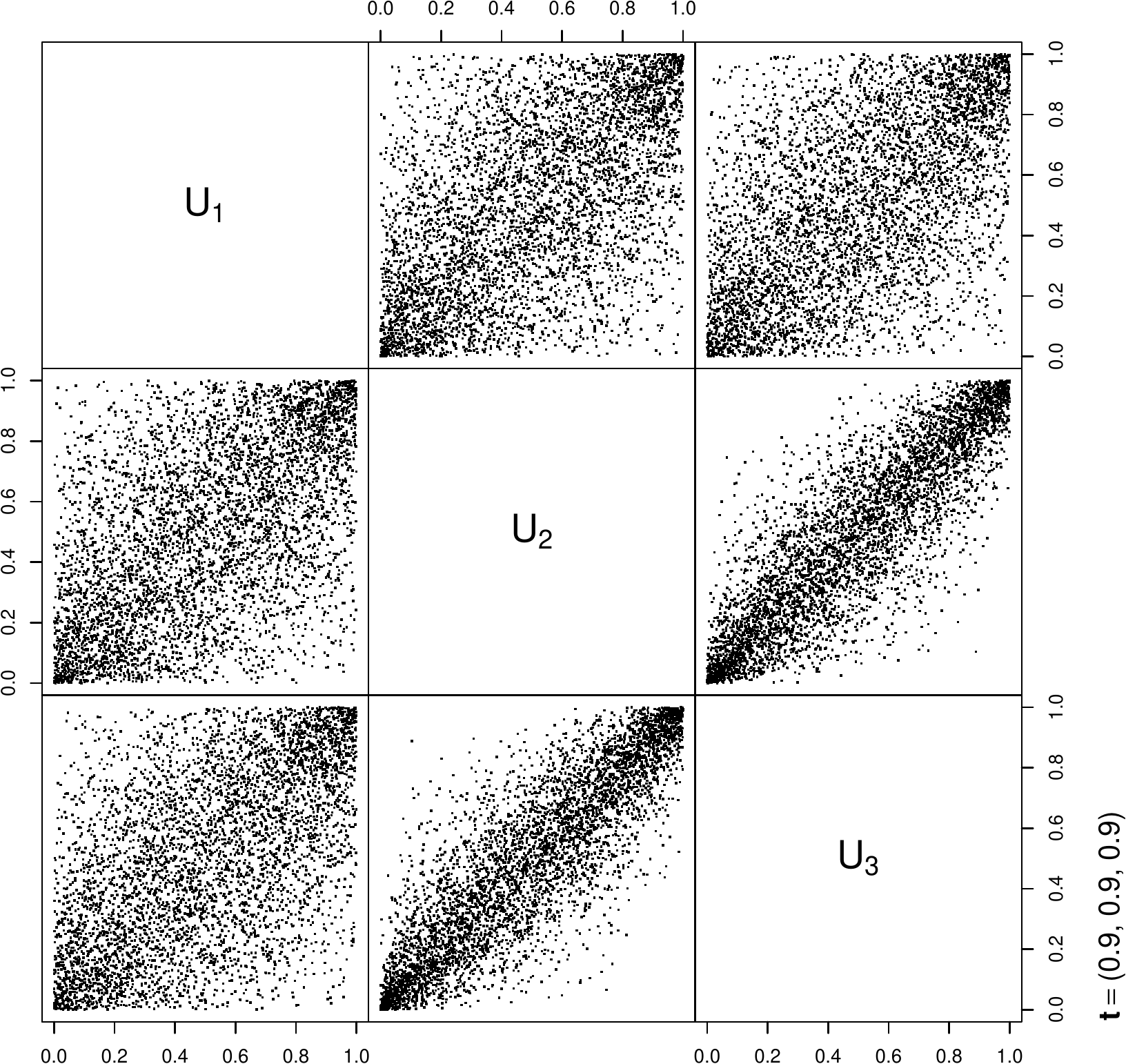}%
    \hspace*{\fill}\par\vspace{2mm}
    \hfill
    \includegraphics[width=0.4\textwidth]{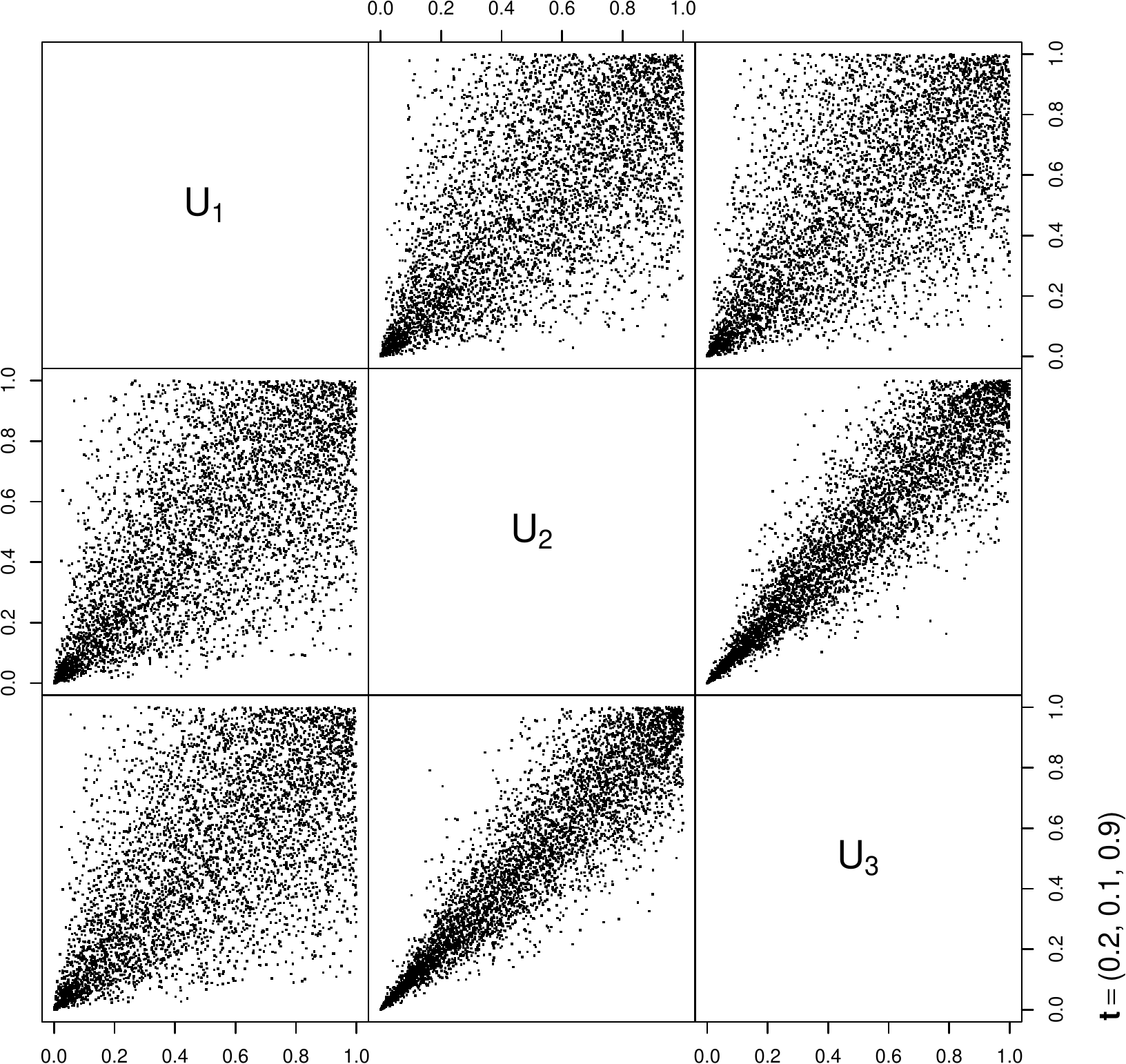}%
    \hfill
    \includegraphics[width=0.4\textwidth]{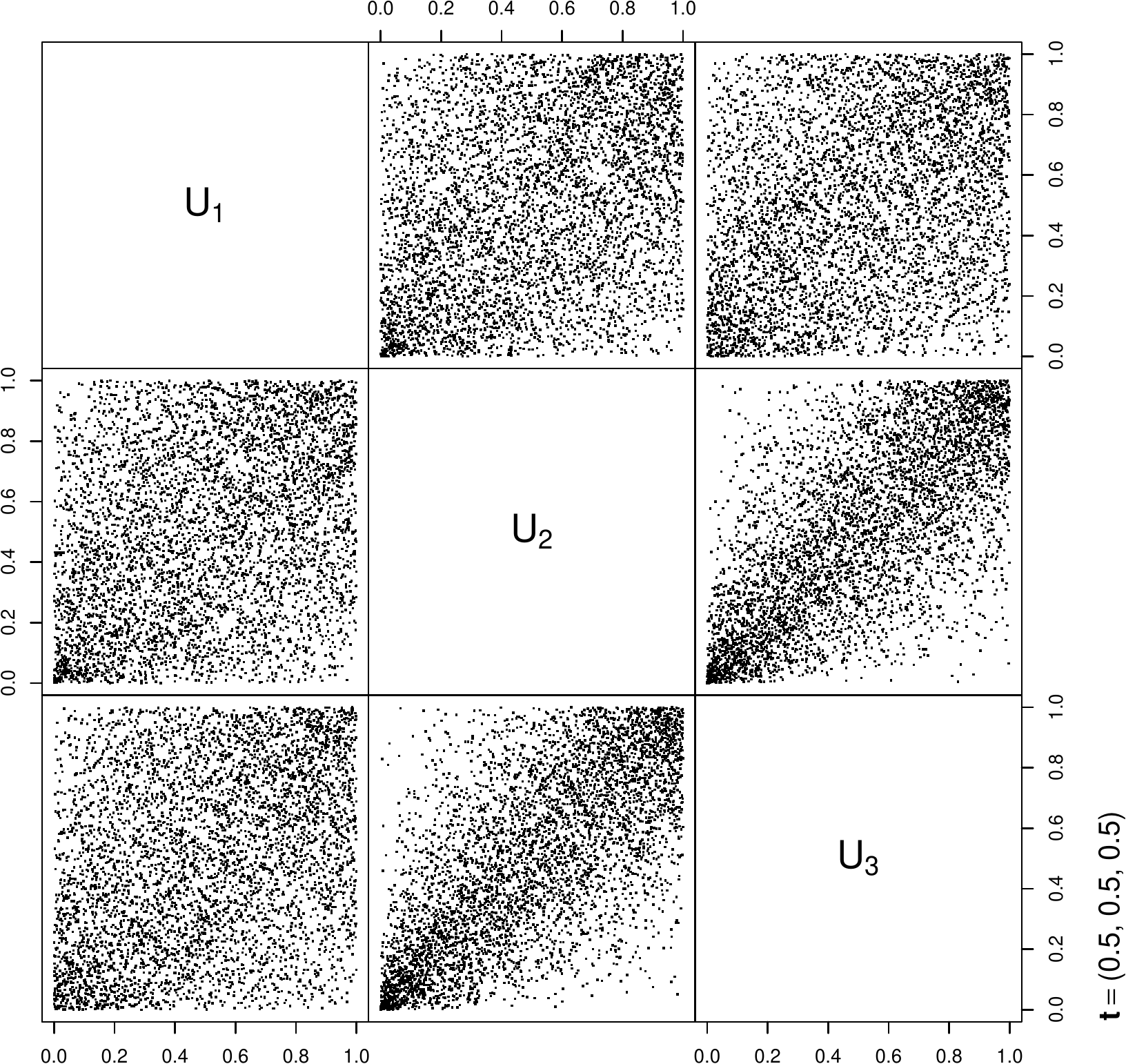}%
    \hspace*{\fill}
    \caption{$n=5000$ pseudo-observations from nested Clayton (left column)
      and nested Gumbel (right column) copulas (of the form $C(\bm{u})=C_0(u_{11},C_1(u_{21},u_{22}))$
      with parameters $\theta_0,\theta_1$ of $C_0,C_1$ chosen such that the
      corresponding Kendall's taus are $0.5, 0.75$, respectively)
      right-truncated at the indicated points $\bm{t}$.}
    \label{fig:rtrunc:NACs}
  \end{figure}
\end{example}

\section{Conclusion}
For $\bm{U}$ following a copula $C$, we considered the copulas $C_{\bm{t}}$ of
$\bm{U}\,|\,\bm{U}\le\bm{t}$, termed right-truncated copulas with truncation
point $\bm{t}=(t_1,\dots,t_d)$. In comparison to the existing literature, we
focused on the case of a fixed truncation point $\bm{t}\in(0,1]^d$, so neither
the case of equal truncation points $\bm{t}=(t,\dots,t)$ nor the limit for
$t\downarrow 0$; some results for the former case can be found in the
appendix. In this setup, we derived a formula for $C_{\bm{t}}$ which is
analytically tractable if $C$ is componentwise analytically invertible; see
Proposition~\ref{prop:rtrunc:C}.  We then considered the case where $\bm{U}$
follows an Archimedean copula and show that the family of copulas of
$\bm{U}\,|\,\bm{U}\le\bm{t}$ can be characterized as tilted Archimedean; see
Theorem~\ref{thm:rtrunc:AC}. For various well-known Archimedean copulas (where a
limiting Clayton model is not be adequate; see
Remark~\ref{rem:limit:clayton:not:adequate}) we were thus able to identify the
corresponding right-truncated copulas; in particular, right-truncated
Ali--Mikhail--Haq or Frank copulas, for example, are again Ali--Mikhail--Haq and
Frank copulas, respectively. We also considered outer power Archimedean copulas
$C$ (or Archimax copulas with logistic stable tail dependence function) and showed
that their right-truncated copulas $C_{\bm{t}}$ are tilted outer power
Archimedean; see Sections~\ref{sec:rtrunc:opAC} and
\ref{sec:rtrunc:AXC:log:stdf}. Furthermore, we derived the right-truncated
copulas of nested Archimedean copulas; see Theorem~\ref{thm:right:rtrunc:NAC}.

One open problem that remains concerns exchangeability of bivariate margins of
right-truncated copulas if not all truncation points are equal; see
Example~\ref{ex:surv:G}.

\appendix
\section{Properties of right-truncated copulas}\label{sec:prop}
In this section we gather selected properties of right-truncated copulas, mainly in the case of equal
truncation points, so $\bm{t}=(t,\dots,t)$ for some $t\in(0,1]$.

The following result provides a sufficient condition for $C_{\bm{t}}$ to be
exchangeable.
\begin{corollary}[Exchangeability]
  If $C$ is exchangeable, that is permutation symmetric in its arguments, and $\bm{t}=(t,\dots,t)$
  for some $t\in(0,1]$, then $F_{\bm{t}}$ and thus $C_{\bm{t}}$ are exchangeable.
\end{corollary}
\begin{proof}
  If $\pi$ denotes a permutation of $\{1,\dots,d\}$ and
  $\bm{x}_{\pi}=(x_{\pi(1)},\dots,x_{\pi(d)})$, then Lemma~\ref{lem:rtrunc:df}
  implies that
  $F_{\bm{t}}(\bm{x}_{\pi})=\frac{C(\min\{x_{\pi(1)},t_1\},\dots,\min\{x_{\pi(d)},t_d\})}{C(t_1,\dots,t_d)}=\frac{C(\min\{x_1,t_{\pi^{-1}(1)}\},\dots,\min\{x_d,t_{\pi^{-1}(d)}\})}{C(t_{\pi^{-1}(1)},\dots,t_{\pi^{-1}(d)})}=F_{\bm{t}_{\pi^{-1}}}(\bm{x})$
  so that if $\bm{t}=(t,\dots,t)$ for some $t\in(0,1]$, $F_{\bm{t}}$ is exchangeable. By
  \cite[Proposition~2.5.5]{hofertkojadinovicmaechleryan2018}, %
  $C_{\bm{t}}$ is exchangeable.
\end{proof}
As we have already seen in Theorem~\ref{thm:rtrunc:AC} for Archimedean copulas (whose
right-truncated copulas are tilted Archimedean copulas and thus exchangeable), exchangeable
copulas $C$ can lead to exchangeable right-truncated copulas $C_{\bm{t}}$ even
for non-equal truncation points. The following example describes an open problem
concerning exchangeability.
\begin{example}[Right-truncated bivariate survival Gumbel copulas]\label{ex:surv:G}
  Figure~\ref{fig:rtrunc:SG:copula:samples} shows 5000 pseudo-observations
  from bivariate survival Gumbel copulas
  $C(u_1,u_2)=-1+u_1+u_2+\psi(\psii[1-u_1]+\psii[1-u_2])$ for
  $\psi(t)=\exp(-t^{1/\theta})$ with parameter $\theta=2$ (Kendall's tau equals
  $0.5$), right-truncated at the truncation points as indicated.
  \begin{figure}[htbp]
    \centering
    \includegraphics[width=0.48\textwidth]{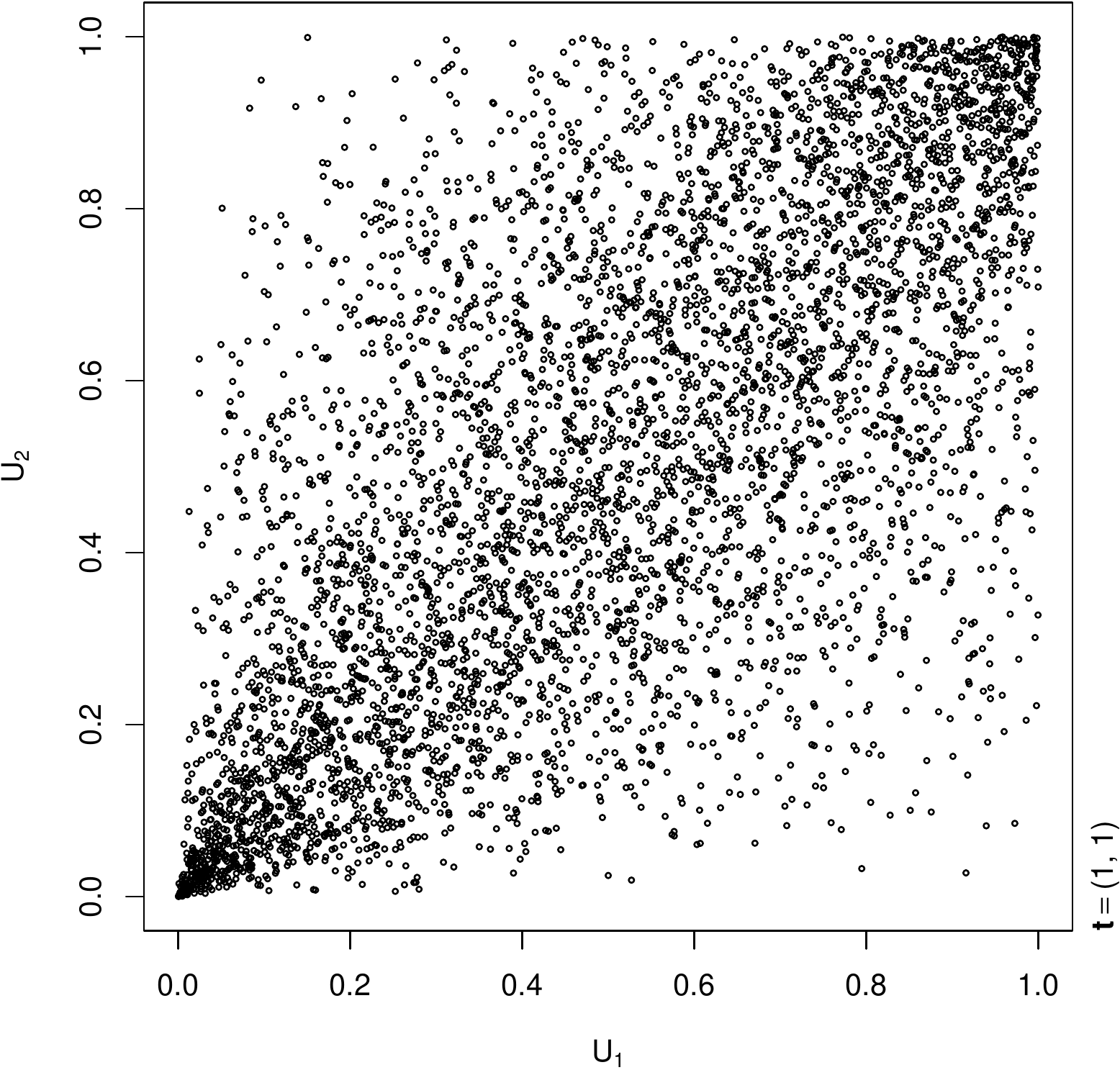}%
    \hfill
    \includegraphics[width=0.48\textwidth]{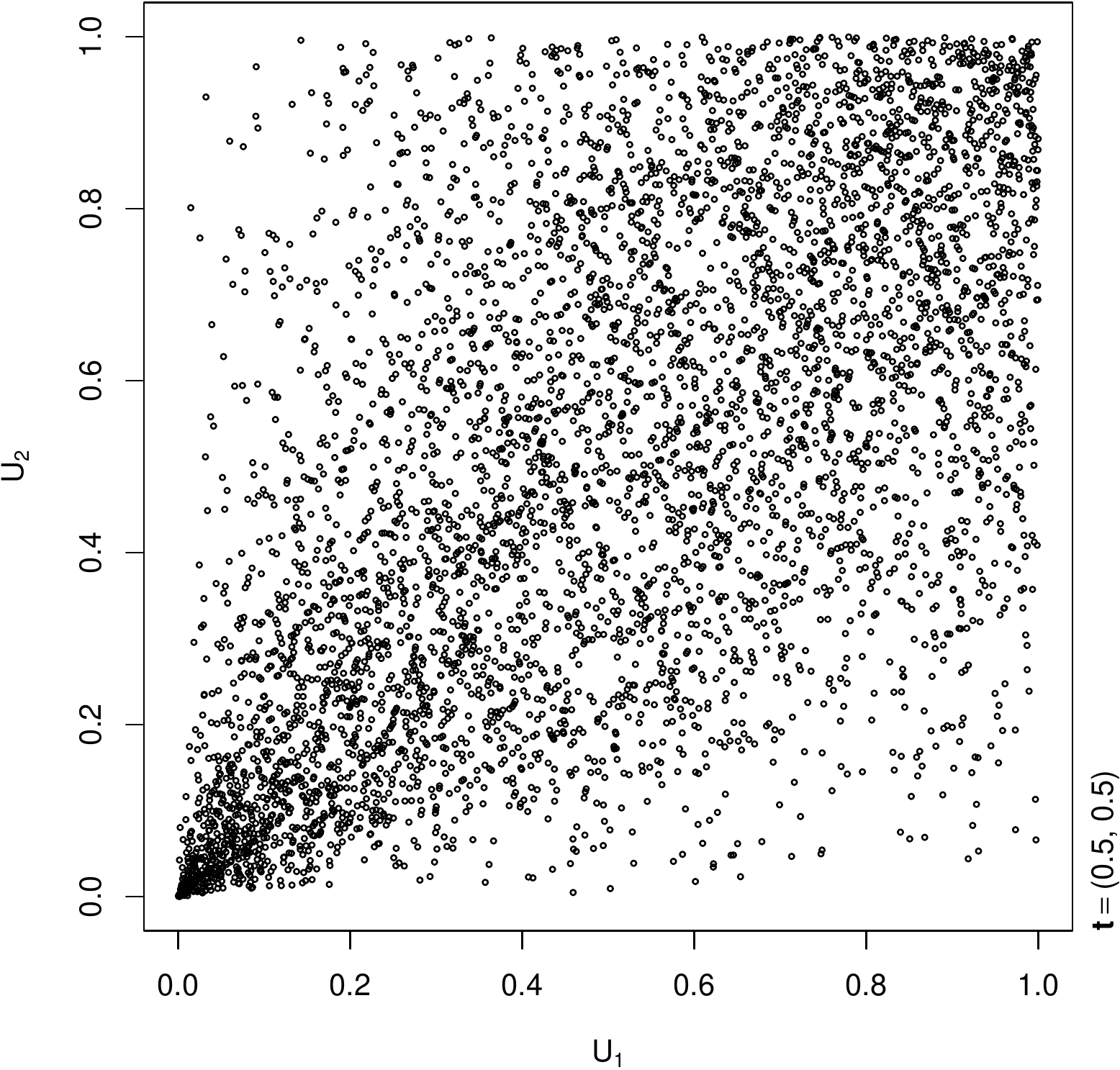}\\[2mm]
    \includegraphics[width=0.48\textwidth]{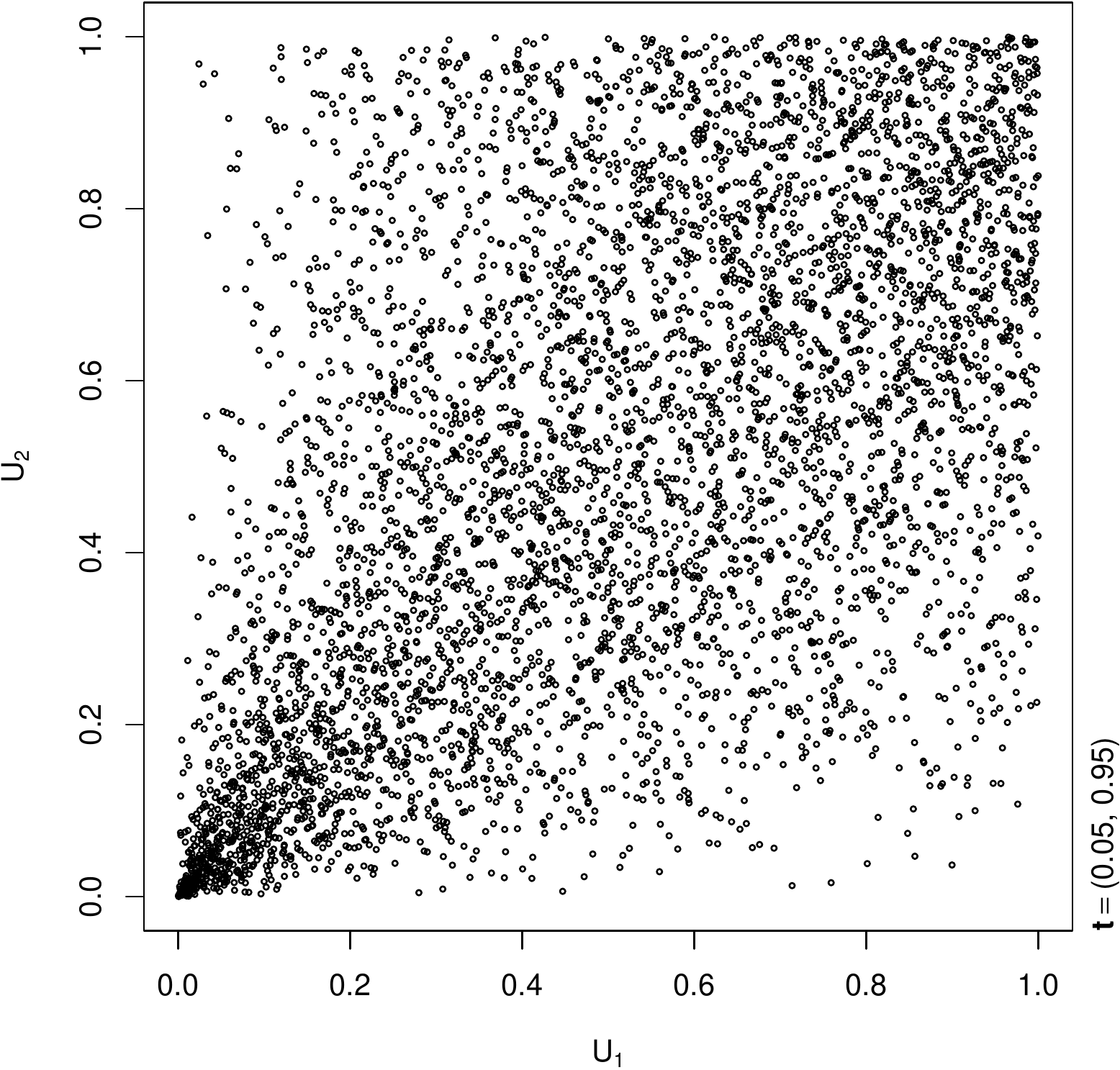}%
    \hfill
    \includegraphics[width=0.48\textwidth]{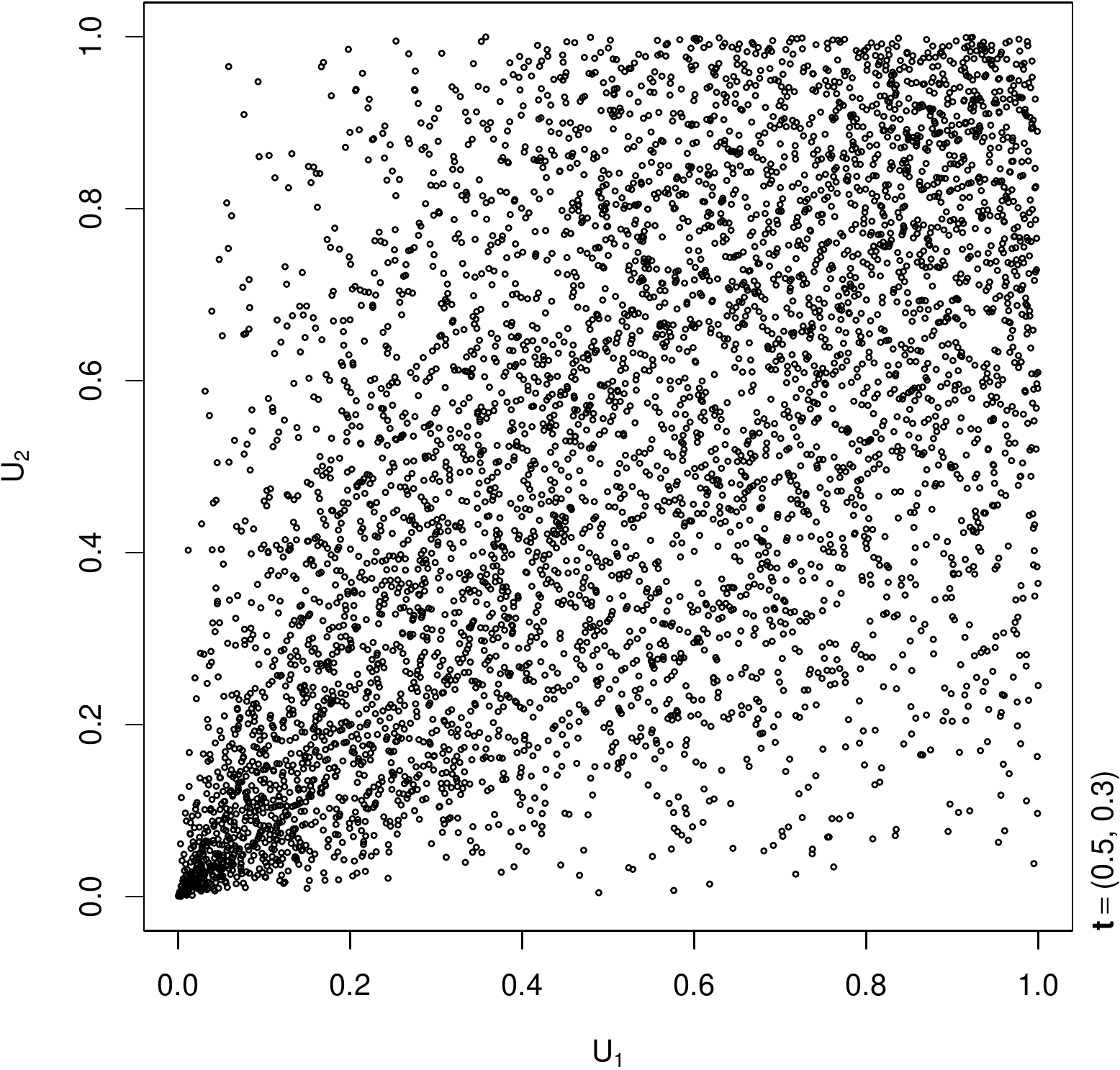}%
    \caption{$n=5000$ pseudo-observations from a survival Gumbel copula (with
      parameter such that Kendall's tau equals $0.5$) right-truncated at the indicated
      points $\bm{t}$.}
    \label{fig:rtrunc:SG:copula:samples}
  \end{figure}
  As we can see from the bottom row of
  Figure~\ref{fig:rtrunc:SG:copula:samples}, it seems that such bivariate copulas
  are exchangeable. It remains an open
  problem to show this property mathematically. In three or more dimensions, non-equal
  truncation points of survival Gumbel %
  copulas lead to non-exchangeable copulas, but each of their bivariate margins again seem to be
  exchangeable; see Figure~\ref{fig:rtrunc:SG:copula:samples:3d}.
  \begin{figure}[htbp]
    \centering
    \includegraphics[width=0.48\textwidth]{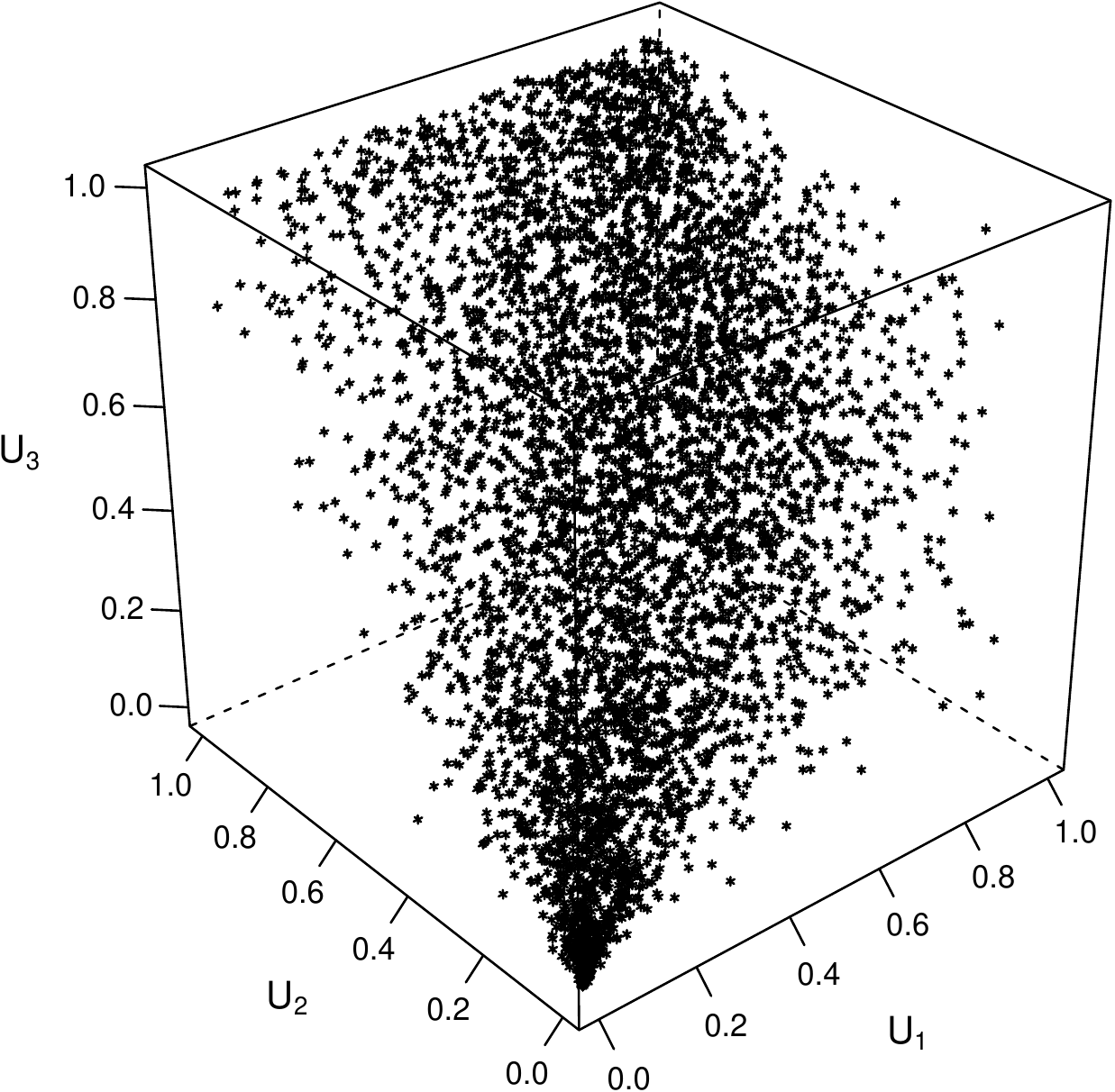}%
    \hfill
    \includegraphics[width=0.48\textwidth]{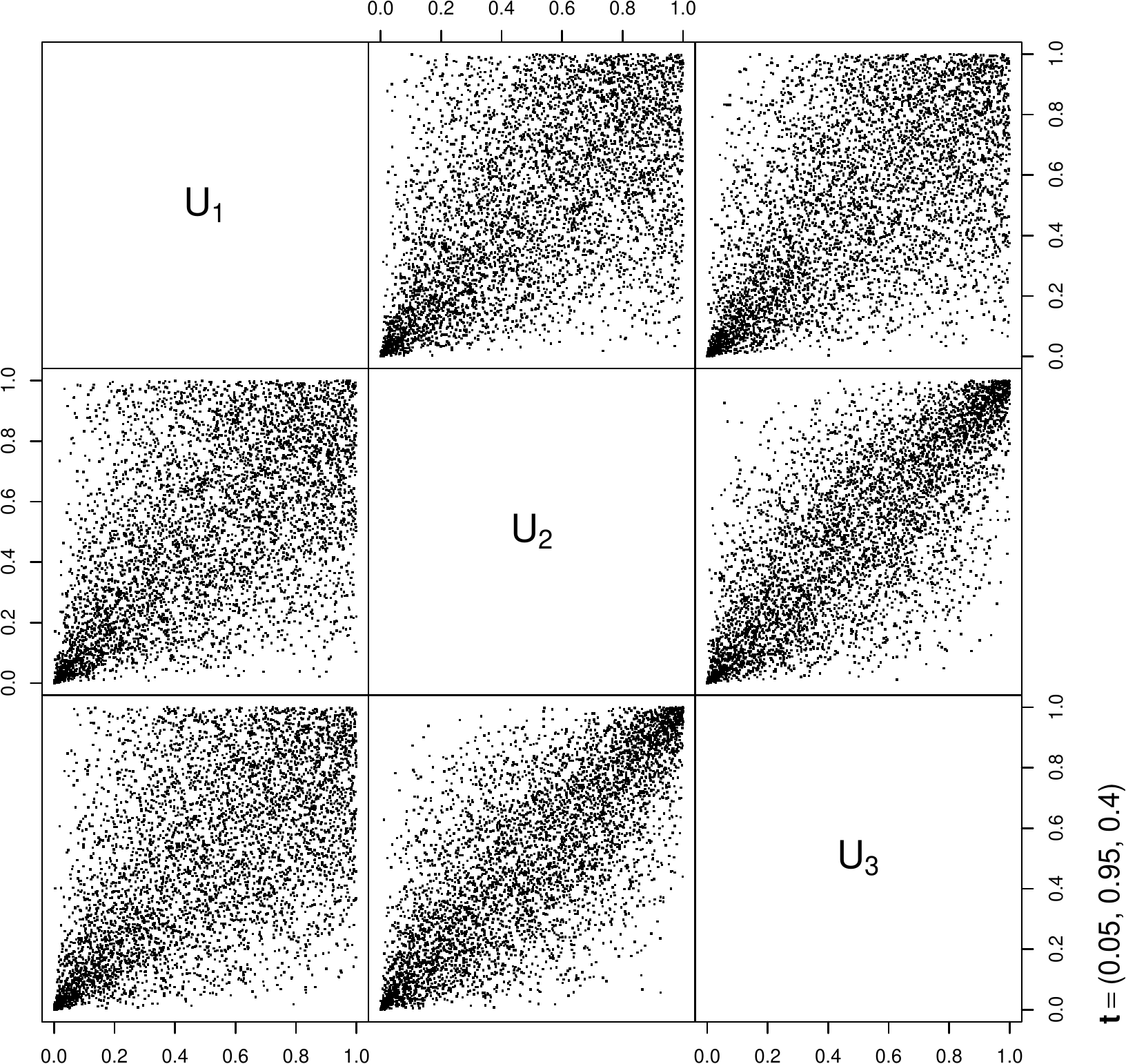}%
    \caption{$n=5000$ pseudo-observations from a trivariate survival Gumbel
      copula (with parameters such that all pairwise Kendall's tau equal $0.5$)
      right-truncated at $\bm{t}=(0.05, 0.95, 0.4)$.}
    \label{fig:rtrunc:SG:copula:samples:3d}
  \end{figure}
\end{example}

The following result addresses the coefficients of tail dependence of right-truncated exchangeable
copulas.
\begin{proposition}[Tail dependence of exchangeable copulas right-truncated at equal truncation points]\label{prop:tail:dep:ex:rtrunc:equal:t}
  Let $C$ be a bivariate exchangeable copula with existing coefficient of lower and upper
  tail dependence $\lambda_{\text{l}}^C$ and $\lambda_{\text{u}}^C$,
  respectively. Furthermore, let $\bm{t}=(t,t)$ for some $t\in(0,1]$ such that $C(\bm{t})>0$.  Assuming
  the limits to exist, the coefficients of lower and upper tail dependence of
  the truncated copula $C_{\bm{t}}$ at $\bm{t}$ are given by
  \begin{align*}
    \lambda_{\text{l}}^{C_{\bm{t}}}=\lim_{u\downarrow 0}\frac{C(u,u)}{C(u,t)} = \lambda_{\text{l}}^C \lim_{u\downarrow 0}\frac{u}{C(u,t)}%
    \quad\text{and}\quad
    \lambda_{\text{u}}^{C_{\bm{t}}}=2-\lim_{u\uparrow t}\frac{C(t,t)-C(u,u)}{C(t,t)-C(u,t)}.%
  \end{align*}
  In particular, $\lambda_{\text{l}}^{C_{\bm{t}}}\ge \lambda_{\text{l}}^{C}$. %

  Furthermore, assuming the limits to exist,
  \begin{align*}
    \lambda_{\text{l}}^{C_{\bm{t}}}=\frac{\lambda_{\text{l}}^{C}}{\D_1 C(0,t)}\quad\text{and}\quad \lambda_{\text{u}}^{C_{\bm{t}}}=2-\frac{\delta_C'(t)}{\D_1C(t,t)},
  \end{align*}
  where $\D_1C(s,t)=\frac{\partial}{\partial u}C(u,t)\big|_{u=s}$ for $s\in[0,t]$ %
  and $\delta_C'(t)=\frac{\partial}{\partial u}\delta_C(u)\big|_{u = t}$ for $\delta_C(u)=C(u,u)$.
\end{proposition}
\begin{proof}
  Assuming the limits to exist, we have
  \begin{align*}
    \lambda_{\text{l}}^{C_{\bm{t}}} &= \lim_{u\downarrow 0}\frac{C_{\bm{t}}(u,u)}{u} = \lim_{u\downarrow 0}\frac{C(C^{-1}[C(\bm{t})u;t],C^{-1}[C(\bm{t})u;t])}{C(\bm{t})u}\\
                                &= \lim_{v\downarrow 0}\frac{C(C^{-1}[v;t],C^{-1}[v;t])}{v} = \lim_{u\downarrow 0}\frac{C(u,u)}{C(u;t)}=\lim_{u\downarrow 0}\frac{C(u,u)}{C(u,t)}. %
  \end{align*}
  If the coefficient of lower tail dependence
  $\lambda_{\text{l}}^C$ of $C$ and the limit $\lim_{u\downarrow 0} \frac{u}{C(u,t)}$
  exist, then
  \begin{align*}
    \lambda_{\text{l}}^{C_{\bm{t}}}=\lim_{u\downarrow 0}\frac{C(u,u)}{u} \lim_{u\downarrow 0}\frac{u}{C(u,t)}= \lambda_{\text{l}}^C \lim_{u\downarrow 0}\frac{u}{C(u,t)}.
  \end{align*}
  To see that $\lambda_{\text{l}}^{C_{\bm{t}}}\ge \lambda_{\text{l}}^{C}$, note that $C(u,t)\le\min\{u,t\}=u$ for $u\le t$.

  Now consider upper tail dependence. Similar as before, we obtain for $C(\bm{t})>0$ that
  \begin{align*}
    \lambda_{\text{u}}^{C_{\bm{t}}}&=\lim_{u\uparrow 1}\frac{1-2u+C_{\bm{t}}(u,u)}{1-u}=2-\lim_{u\uparrow 1}\frac{1-C_{\bm{t}}(u,u)}{1-u}\\
                                   &=2-\lim_{u\uparrow 1}\frac{1-C(C^{-1}[C(\bm{t})u;t],C^{-1}[C(\bm{t})u;t])/C(\bm{t})}{1-u}\\
                                   &= 2-\lim_{v\uparrow C(\bm{t})}\frac{1-C(C^{-1}[v;t],C^{-1}[v;t])/C(\bm{t})}{1-v/C(\bm{t})}\\
                                   &= 2-\lim_{v\uparrow C(\bm{t})}\frac{C(\bm{t})-C(C^{-1}[v;t],C^{-1}[v;t])}{C(\bm{t})-v} = 2-\lim_{u\uparrow t}\frac{C(t,t)-C(u,u)}{C(t,t)-C(u,t)}.
  \end{align*}

  The final formulas for $\lambda_{\text{l}}^{C_{\bm{t}}}$ and $\lambda_{\text{u}}^{C_{\bm{t}}}$ follow from an
  application of l'H\^opital's rule.
\end{proof}

\begin{corollary}[Tail dependence for truncated survival Archimedean copulas]\label{cor:tail:dep:trunc:surv:AC}%
  Let $C$ be the survival copula of an Archimedean copula $C_\psi$ with
  generator $\psi$ satisfying $\psi'(0)=-\infty$; this is the case for all
  Archimedean copulas with completely monotone generators with upper tail dependence;
  see \cite{embrechtshofert2011c}. If $\bm{t}=(t,t)$ for $t\in(0,1]$, then
  $\lambda_{\text{l}}^{C_{\bm{t}}}=\lambda_{\text{u}}^{C_\psi}$ and
  $\lambda_{\text{u}}^{C_{\bm{t}}}=0$.
\end{corollary}
\begin{proof}
  We have $C(u,t)=-1+u+t+\psi(\psii[1-u]+\psii[1-t])$,
  so that
  \begin{align*}
    \D_1 C(u,t)=1-\frac{\psi'(\psii[1-u]+\psii[1-t])}{\psi'(\psii[1-u])}
  \end{align*}
  and thus $\D_1 C(0,t)=1-\psi'(0+\psii[1-t])/\psi'(0)=1$. By
  Proposition~\ref{prop:tail:dep:ex:rtrunc:equal:t},
  $\lambda_{\text{l}}^{C_{\bm{t}}}=\lambda_{\text{l}}^{C}=\lambda_{\text{u}}^{C_\psi}$.
  Furthermore, $\D_1 C(t,t)=1-\psi'(2\psii[1-t])/\psi'(\psii[1-t])$ and
  $\delta_{C}'(t)=2-2\psi'(2\psii[1-t])/\psi'(\psii[1-t])=2\D_1 C(t,t)$
  so that, by Proposition~\ref{prop:tail:dep:ex:rtrunc:equal:t},
  $\lambda_{\text{u}}^{C_{\bm{t}}}=2-2=0$.
\end{proof}
For example, for the survival Gumbel copula $C$ of Example~\ref{ex:surv:G}, we
obtain from Corollary~\ref{cor:tail:dep:trunc:surv:AC} that the corresponding
right-truncated copula $C_{\bm{t}}$ with equal truncation points has coefficients of lower
and upper tail dependence given by
$\lambda_{\text{l}}^{C_{\bm{t}}}=\lambda_{\text{u}}^{C_\psi}=2-2^{1/\theta}=2-\sqrt{2}\approx
0.5858$ and $\lambda_{\text{u}}^{C_{\bm{t}}}=0$, respectively; compare with the
top row of Figure~\ref{fig:rtrunc:SG:copula:samples}.

\printbibliography[heading=bibintoc]
\end{document}

%
%
%
%

%%% Local Variables:
%%% mode: latex
%%% TeX-master: t
%%% End: